\documentclass[11pt]{amsart}
\usepackage{mathrsfs}

\usepackage{latexsym}
\usepackage{amssymb}
\usepackage{amsfonts}

\usepackage{amscd}
\usepackage{bbm}
\usepackage{mathabx}
\usepackage{skak}

\usepackage{color}

\usepackage{hyperref}
\hypersetup{
    colorlinks,
    citecolor=black,
    filecolor=black,
    linkcolor=black,
    urlcolor=black
}

\usepackage[all]{xy}

\newtheorem{proposition}{Proposition}[section]

\newtheorem{lemma}[proposition]{Lemma}
\newtheorem{definition}[proposition]{Definition}
\newtheorem{theorem}[proposition]{Theorem}

\newtheorem{conjecture}[proposition]{Conjecture}

\newtheorem{corollary}[proposition]{Corollary}

\newtheorem{example}[proposition]{Example}
\newtheorem{remark}[proposition]{Remark}

\newtheorem{lemma-definition}[proposition]{Lemma-Definition}

\newtheorem{question}[proposition]{Question}
\newtheorem{problem}[proposition]{Problem}

\newcounter{tmp}

\def\Qcoh{\operatorname{Qcoh}}

\def\lto{\longrightarrow}

\def\A{{\mathcal A}}
\def\B{{\mathcal B}}

\def\D{{\mathcal D}}

\def\H{{\mathcal H}}

\def\N{{\mathcal N}}
\def\cO{{\mathcal O}}
\def\R{{\mathcal R}}
\def\L{{\mathcal L}}

\def\T{{\mathcal T}}
\def\I{{\mathcal I}}

\def\ZZ{{\mathbb Z}}

\def\bR{{\mathbf R}}

\def\bp{{\mathbf p}}
\def\bL{{\mathbf L}}

\def\AA{{\mathbb A}}
\def\NN{{\mathbb N}}
\def\ZZ{{\mathbb Z}}

\def\PP{{\mathbb P}}

\def\Hom{\operatorname{Hom}}
\def\End{\operatorname{End}}
\def\Ext{\operatorname{Ext}}

\def\Int{\operatorname{Int}}

\def\Cone{\operatorname{Cone}}

\def\Spec{\operatorname{Spec}}

\def\kk{{\Bbbk}}
\def\ff{{\mathbb F}}

\def\op{\circ}

\newcommand{\Ho}{{\H^0}}

\newcommand{\Ob}{\operatorname{Ob}}

\newcommand{\SF}{\dS\!\dF\!\operatorname{--}\!}
\newcommand{\SFf}{\dS\!\dF_{fg}\!\operatorname{--}\!}
\newcommand{\prfdg}{\mathscr{P}\!\mathit{erf}\!\operatorname{--}}

\newcommand{\Ac}{\dA\!\mathit{c}\!\operatorname{--}\!}

\newcommand{\prf}{\operatorname{perf}\!\operatorname{--}}

\def\dA{\mathscr A}
\def\dB{\mathscr B}
\def\dC{\mathscr C}

\def\dE{\mathscr E}

\def\dF{\mathscr F}

\def\dH{\mathscr H}
\def\dI{\mathscr I}

\def\dK{\mathscr K}
\def\dL{\mathscr L}

\def\dN{\mathscr N}
\def\dP{\mathscr P}
\def\dR{\mathscr R}
\def\dS{\mathscr S}

\def\dX{\mathscr X}
\def\dY{\mathscr Y}
\def\dZ{\mathscr Z}

\def\Mod{{\mathscr M}\!\mathit{od}\!\operatorname{--}\!}

\def\mE{\mathsf E}

\def\mF{\mathsf F}
\def\mI{\mathsf I}
\def\mJ{\mathsf J}
\def\mK{\mathsf K}
\def\mL{\mathsf L}
\def\mM{\mathsf M}
\def\mN{\mathsf N}
\def\mP{\mathsf P}

\def\mR{\mathsf R}

\def\mT{\mathsf T}

\def\mh{\mathsf h}
\def\mf{\mathsf f}

\def\mPhi{\mathsf \Phi}

\def\m0{\mathsf 0}

\def\mX{X}
\def\mY{Y}
\def\mZ{Z}

\def\dHom{\mathsf{Hom}}
\def\dEnd{\mathsf{End}}

\def\ptr{\operatorname{pre-tr}}

\def\Rep{{\R\!\mathit{ep}}}

\def\rd{{J}}
\def\rdi{\mJ_{-}}
\def\rde{\mJ_{+}}

{\endgroup\hfill$\Box$}

{\endgroup\hfill$\Box$}

\def\hy{\mbox{-}}

\def\gE{E}
\def\gR{R}
\def\gS{S}
\def\gM{M}
\def\gN{N}
\def\gHom{\Hom}

\def\un{\underline}

\def\La{\Lambda}

\def\Aut{\operatorname{\mathbf{Aut}}}

\def\dr{d_{\dR}}
\def\ds{d_{\dS}}

\def\cf{\mathrm{cf}}

\title[]{Finite-dimensional differential graded algebras and their geometric realizations}

\author[]{Dmitri Orlov}

\address{ Algebraic Geometry Dept., Steklov Math. Institute RAS,
8 Gubkin str., Moscow 119991, RUSSIA}
\email{orlov@mi-ras.ru}

\thanks{This work is supported by the Russian Science Foundation under  grant 19-11-00164}

\date{}


\begin{document}

\begin{abstract}
We prove that
for any finite-dimensional differential graded algebra with separable semisimple part
the category of perfect modules is equivalent to a full subcategory of the category of perfect complexes on
a smooth projective scheme with a full separable semi-exceptional collection.
Moreover, we also show that it gives a characterization of such categories assuming that a subcategory is idempotent complete and  has a classical generator.
\end{abstract}

\maketitle

\section*{Introduction}

The main goal of this article is to study and describe the world of derived noncommutative schemes, which are related to finite-dimensional differential graded (DG) $\kk$\!--algebras.
By a derived noncommutative scheme $ \dX $ over $\kk$ we mean a $\kk$\!--linear differential graded (DG) category of perfect modules $ \prfdg \dR$ over a cohomologically bounded DG $\kk$\!--algebra $\dR,$
i.e. such a DG algebra that has only a finite number of nontrivial cohomology spaces
(see , e.g., \cite{O_glue, O_alg}).
The reason of such a definition is a fact that the DG category
of perfect complexes  on any usual commutative quasi-compact and quasi-separated scheme $X$ is quasi-equivalent to a DG category of the form
$ \prfdg \dR, $  where $ \dR $ is a cohomologically bounded DG $\kk$\!--algebra. It follows from results of the papers \cite{Ne, BV, Ke, Ke2} and, in this case, we also obtain
that  the derived category $ \D_{\Qcoh} (X) $ of unbounded complexes of sheaves of $\cO_X$\!--modules with quasi-coherent cohomology  is equivalent to the derived
category $\D(\dR)$ of all DG modules
 over the DG algebra  $\dR.$
The differential graded algebra $\dR$ depends on a choice of a classical generator $\mE \in \prf X$ and it appears as the  DG algebra of endomorphisms
$\dEnd_{\prfdg X} (\mE)$ of this classical generator $\mE$ in the DG category $\prfdg X.$
The derived category $ \D(\dR) $ will be called by the derived category of quasi-coherent sheaves on the derived noncommutative scheme $\dX,$
while the triangulated category $ \prf \dR $ will be called the category of perfect complexes on it.

Many important properties of usual commutative schemes can be extended to derived noncommutative schemes.
For example, one can define when derived noncommutative schemes are proper, smooth or regular.
In particular, a derived noncommutative scheme $\dX=\prfdg\dR$ over $\kk$ will be {\sf proper} if and only if
the cohomology $\kk$\!--algebra $\bigoplus_{p\in\ZZ}H^p(\dR)$ is finite-dimensional.
One of the most natural classes of proper derived noncommutative schemes is the class of such schemes  $\dX=\prfdg\dR,$ for which DG $\kk$\!--algebras $\dR$ themselves are finite-dimensional.
In this paper we will find  some characterization for this class of  derived noncommutative schemes.

One of the most important properties of both commutative and noncommutative schemes is the properties of smoothness and regularity.
A derived noncommutative scheme $\dX=\prfdg\dR$ will be called {\sf regular} if the triangulated category $\prf\dR$ has a strong generator, i.e.  any classical generator generates the whole category
$\prf\dR$ in a finite number of steps
(Definitions \ref{regprop} and \ref{propdef}).
The notion of smoothness is well-defined for any DG $\kk$\!--algebra (and any DG $\kk$\!--linear category). A DG $\kk$\!--algebra $\dR$ is called {\sf $\kk$\!--smooth} if it is perfect as the bimodule over
itself, i.e. as the module over the DG algebra $\dR^{\op}\otimes_{\kk}\dR.$ It also can be shown that the property for DG algebra $\dR$ to be smooth is the property
of the derived noncommutative scheme $\dX=\prfdg\dR$ and does not depend on the choosing of a DG algebra $\dR.$

Similar to the commutative case in which the simplest schemes are $\Spec\ff,$ where $\ff\supseteq\kk$ is a finite  field extension, in a noncommutative situation, the simplest schemes are
$\prfdg D,$ where $D$ is a finite-dimensional division $\kk$\!--algebra. Note that such schemes are always regular and they are $\kk$\!--smooth if and only if they are separable over $\kk.$
On the other hand, in the world of derived noncommutative schemes there is a remarkable gluing operation
 that enables to build new derived noncommutative schemes and which has no any analogue in the commutative case.
More precisely, having two DG categories $ \dA $ and $ \dB $
 and a $ \dB^{\op}\hy\dA $\!--bimodule $ \mT$ one can construct a new DG category $\dC = \dA \underset{\mT} {\oright} \dB$
   that will be called a gluing of the DG categories $ \dA $ and $ \dB $ with respect to the bimodule $ \mT. $
In such way, assuming that the bimodule $ \mT$ is cohomologically bounded, we can define a gluing of any derived noncommutative schemes $ \dX $ and $ \dY.$
Derived noncommutative schemes of the form $ \dX \underset {\mT} {\oright} \dY $ inherit many different properties of the
 noncommutative schemes $\dX$ and $\dY.$
For example, in the case when a bimodule $ \mT $ is perfect, the smoothness and properness of $ \dX $ and $ \dY $ imply the smoothness and properness for the noncommutative scheme
$ \dX \underset {\mT} {\oright} \dY.$
A gluing procedure can be iterated and it allows to reproduce new and new noncommutative schemes.

Starting form the set of the simplest noncommutative schemes of the form $\prfdg D,$ where $D$ are separable finite-dimensional division $\kk$\!--algebras, and using
the gluing operation via perfect bimodules, we can obtain a large class of proper and smooth derived noncommutative schemes. For any derived noncommutative scheme $\dX=\prfdg\dR$ from this class
its triangulated category of perfect modules $\prf\dR$ has a full semi-exceptional collection (see Definition \ref{semi-exc}).
Moreover, it is also well-known and is almost evident that if the triangulated category $\prf\dR$ has a full semi-exceptional collection then the derived noncommutative scheme
$\dX=\prfdg\dR$ can be obtained by a gluing of  schemes of the form $\prfdg D,$ where $D$ are finite-dimensional division algebras.
If, in addition, the derived noncommutative scheme  $\dX$ is smooth and proper, then all algebras $D$ have to be separable and the gluing procedure is carried out  using only perfect bimodules.

In such way we obtain a certain class of smooth and proper derived noncommutative schemes which will be called a class of schemes with a full separable semi-exceptional collection.
Any such derived noncommutative scheme with a full separable semi-exceptional collection itself has a form $\prfdg\dE,$ where $\dE$ is a smooth finite-dimensional DG $\kk$\!--algebra (see Corollary \ref{sep_coll}).
On the other hand, there are a lot of finite-dimensional DG $\kk$\!--algebras $\dR$ for which the derived noncommutative schemes $\prfdg\dR$ does not have  full semi-exceptional collections even if they are smooth.
However,  it can be shown that finite-dimensional DG $\kk$\!--algebras are directly related to the class of derived noncommutative schemes with  full separable semi-exceptional collections.

One of the main two purpose  of this paper is to suggest a generalization and an extension of the famous Auslander construction for finite-dimensional algebras  to the case of finite-dimensional DG algebras.
More precisely, for any finite-dimensional DG algebra $\dR$ we construct another finite-dimensional DG algebra $\dE$ the category $\prf\dE$ of which contains $\prf\dR$ as a full subcategory on the one hand and it has a full semi-exceptional collection on the other (see Theorem \ref{DGinclus}).  The constructed derived noncommutative scheme $\prfdg\dE$ is regular and it provides a regular resolution of singularities for
the derived noncommutative scheme $\prfdg\dR.$ In addition, if the DG $\kk$\!--algebra $\dR$ has a separable semisimple part, then the DG $\kk$\!--algebra $\dE$ is also smooth and it gives a smooth resolution of singularities for $\prfdg\dR.$
Furthermore, this construction gives us a complete characterization of all derived noncommutative schemes $\dX=\prfdg\dR$
with a finite-dimensional DG algebra $\dR.$ One can show that any full DG subcategory $\dA\subset\prfdg\dE$ of a DG category $\prfdg\dE$ with a full separable semi-exceptional collection has the such form if and only if $\H^0(\dA)$  is idempotent complete and has a classical generator (see Corollary \ref{DGinclus}). In addition, $\dA\subset\prfdg\dE$ will be admissible if and only if the DG algebra $\dR$ is smooth.

The last section is devoted to the study of geometric realizations of finite-dimensional DG algebras.
We prove that it exists if the semisimple part is separable. More precisely, we show that for any
finite-dimensional DG algebra $\dR$ with separable semisimple part there are a smooth projective scheme
$V$ and a perfect complex $\mE\in \prfdg V$ such that the DG algebra $\dEnd_{\prfdg V}(\mE)$  is quasi-isomorphic to $\dR$
(see Theorem \ref{DGinclus_smooth}).
This fact implies that the DG category $\prfdg\dR$ is quasi-equivalent to a full DG subcategory of the DG category $\prfdg V.$
Moreover, by construction, we show that the variety $V$ has a full separable semi-exceptional collection.
These results allow us to embed any derived noncommutative scheme $\dX=\prfdg\dR$ to the geometric DG category
$\prfdg V$ for a smooth projective scheme with a full separable semi-exceptional collection. This gives us a geometric realization for any such derive noncommutative scheme and it also provides another characterization for the class of such schemes as full DG subcategories $\dA\subset\prfdg V$ for which $\H^0(\dA)$  is idempotent complete and has a classical generator, while $V$ has a full separable semi-exceptional collection (see Corollary \ref{DGinclus_smooth}).
In the end of the paper  we formulate two conjectures, pose a couple of questions, and consider some simple examples.

The author is very grateful to Alexander Efimov, Anton Fonarev, Sergei Gorchinskiy, Mikhail Khovanov, and Alexander Kuznetsov for very useful discussions and valuable comments.

\section{Preliminaries on differential graded algebras and categories}

\subsection{Differential graded algebras and categories}

Let $\kk$  be a field. Recall that  a {\sf differential graded $\kk$\!--algebra (=DG
algebra)} $\dR=(\gR, \dr)$ is  a $\ZZ$\!--graded associative $\kk$\!-algebra
\[
\gR =\bigoplus_{q\in \ZZ} R^q
\]
endowed with a $\kk$\!-linear differential $\dr: \gR \to \gR$  (i.e. homogeneous
maps $\dr$ of degree 1 with $\dr^2 = 0$) that satisfies the graded Leibniz rule
\[
\dr(xy) = \dr (x) y + (-1)^q x \dr (y) ,\quad \text{for all}\quad  x\in R^q, y\in \gR.
\]
We consider DG algebras with identity $1\in R^0$ and $\dr(1)=0.$

Any ordinary $\kk$\!--algebra $\La$ gives rise to a DG algebra $\dR$ defined by
$R^0=\La$ and $R^q = 0$ when $q \ne 0.$
Conversely, any DG algebra $\dR$ which is concentrated in degree $0$ is obtained
in this way from an ordinary algebra.

A {\sf differential graded module $\mM$ over  $\dR$ (=DG $\dR$\!--module)} is a $\ZZ$\!-graded right
$\gR$\!-module
\[
\gM = \bigoplus_{q\in\ZZ} M^q
\]
endowed with a $\kk$\!-linear differential $d_{\mM}: \gM \to \gM$ of degree 1 for which $d_{\mM}^2=0$ and
the graded Leibniz rule holds, i.e.
\[
d_{\mM}(mr) = d_{\mM}(m) r + (-1)^q m \dr( r) , \quad\text{for all}\quad m\in M^q, r\in \gR.
\]

A morphism of DG $\dR$\!-modules $f: \mM \to \mN$ is a morphism of the underlying $\gR$\!--modules
$\gM$ and $\gN$
which is homogeneous of degree $0$ and commutes with the differential.
When a DG algebra $\dR$ is an ordinary algebra $\La,$ then the category of DG $\dR$\!--modules identifies with
the category of differential complexes of right $\La$\!--modules.

If $\dR$ is a DG algebra and
$\mM$ and $\mN$ are two DG modules then we can define a complex of $\kk$\!--vector spaces $\dHom_{\dR} (\mM, \mN).$
As a graded vector space it is equal to the graded space of homomorphisms
\begin{equation}\label{grHom}
\gHom_{\gR}^{gr} (\gM, \gN):=\bigoplus_{q\in\ZZ}\Hom_{\gR} (\gM, \gN)^q,
\end{equation}
where $\Hom_{\gR} (\gM, \gN)^q$ is the space of homogeneous homomorphism of $\gR$\!--modules of degree $q.$
The differential $D$ on $\dHom_{\dR} (\mM, \mN)$ is defined by the following rule
\begin{equation}\label{DHom}
D(f) = d_{\mN} \circ f - (-1)^q
f\circ d_{\mM}\quad\text{ for each}\quad f\in \Hom_{\gR} (\gM, \gN)^q.
\end{equation}
It is easy to see that $\dHom_{\dR} (\dR, \mN)\cong\mN$ for any DG modules $\mN.$

A DG algebra is a particular case of differential graded (DG) category. In fact it is a DG category with a single object.
A {\sf differential graded (DG) category} is a $\kk$\!--linear category $\dA$ whose morphism spaces $\dHom (\mX, \mY)$
are complexes of $\kk$\!-vector spaces (DG $\kk$\!--modules), so that for any $\mX, \mY, \mZ\in
\Ob\dC$ the composition $\dHom (\mY, \mZ)\otimes \dHom (\mX, \mY)\to
\dHom (\mX, \mZ)$ is a morphism of DG $\kk$\!--modules. The identity morphism $1_\mX\in \dHom (\mX, \mX)$ is closed of
degree zero.

To any DG category $\dA$ we associate its {\sf homotopy category} $\Ho(\dA),$ which, by definition,
has the same objects as the DG category $\dA$ while
morphisms in $\Ho(\dA)$ are obtained by taking the $0$\!--th cohomology
$H^0(\dHom_{\dA} (\mX, \mY))$ of the complexes of morphisms $\dHom_{\dA} (\mX, \mY).$

In a natural way, a {\sf DG functor}
$\mF:\dA\to\dB$ is given by a map $\mF:\Ob(\dA)\to\Ob(\dB)$ and
by maps of DG $\kk$\!--modules
\[
\mF_{\mX, \mY}: \dHom_{\dA}(\mX, \mY) \lto \dHom_{\dB}(\mF \mX,\mF \mY),\quad \mX, \mY\in\Ob(\dA),
\]
which preserve the compositions and the identity morphisms.

A DG functor $\mF: \dA\to\dB$ will be called a {\sf quasi-equivalence} if all maps
$\mF_{\mX, \mY}$ are  quasi-isomorphisms for any pairs of objects $\mX, \mY\in \dA$ while
the induced homotopy functor $H^0(\mF): \Ho(\dA)\to \Ho(\dB)$ is an
equivalence. Two DG categories $\dA$ and $\dB$ will be called {\sf quasi-equivalent}, if there is a collection of DG
categories $\dC_1,\dots, \dC_m$ and a sequence of quasi-equivalences
$\dA\stackrel{\sim}{\leftarrow} \dC_1 \stackrel{\sim}{\rightarrow} \cdots \stackrel{\sim}{\leftarrow} \dC_m
\stackrel{\sim}{\rightarrow} \dB.$

\subsection{Differential graded modules over differential graded categories}

Given a small DG category $\dA$ we define a {\sf (right) DG $\dA$\!--module} as a DG functor
$\mM: \dA^{op}\to \Mod \kk,$ where $\Mod \kk$ is the DG category of DG $\kk$\!--modules. All (right) DG $\dA$\!--modules form a DG category which we denote by  $\Mod \dA.$
Let $\Ac\dA$ be the full
DG subcategory of $\Mod\dA$ consisting of acyclic DG $\dA$\!--modules, i.e. such DG modules $\mM$
that the complex $\mM(\mX)$ is acyclic for all $X\in\dA.$
The
homotopy category of DG modules $\Ho(\Mod\dA)$ has a natural structure of a triangulated category
and the homotopy subcategory of acyclic complexes $\Ho (\Ac\dA)$ forms a full triangulated subcategory in it.
The {\sf derived
category} $\D(\dA)$ is defined as the Verdier quotient
\[
\D(\dA):=\Ho(\Mod\dA)/\Ho (\Ac\dA).
\]

Each object $\mY\in\dA$ produces a right DG $\dA$\!--module
$
\mh^\mY(-):=\dHom_{\dA}(-, \mY)
$
which is called a {\sf representable} DG module. The
natural DG functor
$\mh^\bullet :\dA \to
\Mod\dA$ is full and faithful and called the Yoneda DG functor.
A DG $\dA$\!--module is called {\sf free} if it is isomorphic to a direct sum of  DG modules of the form
$\mh^\mY[n],$ where $\mY\in\dA,\; n\in\ZZ.$
A DG $\dA$\!--module
$\mP$ is called {\sf semi-free} if it has a filtration
$0=\mPhi_0\subset \mPhi_1\subset ...=\mP$
such that each quotient  $\mPhi_{i+1}/\mPhi_i$ is free. We denote by $\SF\dA$ the full
DG subcategory of $\Mod\dA$ which consists of all semi-free DG $\dA$\!--modules 
We also consider the full DG subcategory $\SFf\dA\subset \SF\dA$  of all finitely generated semi-free
DG modules, i.e. such that $\mPhi_m=\mP$ for some $m$ and $\mPhi_{i+1}/\mPhi_i$ is a finite direct sum of DG modules of the form
$\mh^Y[n]$ for any $i.$
For any DG
$\dA$\!--module $\mM$ we can found  a semi-free
DG $\dA$\!--module $\bp M$ with a quasi-isomorphism $\bp \mM\to \mM.$

It is natural to consider a DG category of so called h-projective (homotopically projective) DG $\dA$\!--modules.
A  DG $\dA$\!--module $\mP$ is called {\sf h-projective } if the complex
$
\dHom_{\Mod\dA}(\mP, \mN)=0
$
is acyclic for any acyclic DG $\dA$\!--module $\mN$ (dually, one defines {\sf h-injective} DG modules).
Denote by  $\dP(\dA)\subset \Mod\dA$ the full
DG subcategory of h-projective DG $\dA$\!--modules. Any semi-free
DG $\dA$\!--module is  h-projective and the natural inclusion of DG subcategories $\SF\dA\hookrightarrow\dP(\dA)$
is a quasi-equivalence.
The canonical DG functors $\SF\dA\hookrightarrow\dP(\dA)\hookrightarrow\Mod\dA$ induce equivalences
$\Ho(\SF\dA)\stackrel{\sim}{\to} \Ho(\dP(\dA))\stackrel{\sim}{\to} \D(\dA)$ between the triangulated categories.

\begin{definition} The DG category of {\sf perfect DG modules} $\prfdg\dA$
is the full DG subcategory of $\dP(\dA)$ consisting of all DG $\dA$\!--modules which are isomorphic  to  direct summands of objects from $\SFf\dA$
in the homotopy category $\Ho(\dP(\dA))\cong\D(\dA).$
\end{definition}

The homotopy category $\Ho(\prfdg\dA)$ is usually denoted by $\prf\dA.$  It is equivalent to the triangulated subcategory
 of compact objects $\D(\dA)^c\subset \D(\dA)$ (see, e.g., \cite{Ne, Ke2}).

Let $\dA$ and $\dB$ be two DG categories and let $\mT$ be a DG
$\dA\hy\dB$\!--bimodule.
For any DG $\dA$\!--module $\mM$ we can produce a DG $\dB$\!--module
$\mM\otimes_{\dA} \mT.$ Thus we obtain
a DG functor $(-)\otimes_{\dA} \mT: \Mod\dA \to \Mod\dB.$ It has a right adjoint DG functor
$\dHom_{\dB} (\mT, -).$
This pair of DG functors do not respect
quasi-isomorphisms, in general. Applying them to h-projective and h-injective DG modules
we can define  derived functors
\begin{align}\label{derived_quasi}
\bL F^*\cong (-)\stackrel{\bL}{\otimes}_\dA \mT: \D(\dA) \lto \D(\dB)
\quad
\text{and}
\quad
\bR F_*\cong \bR\Hom_{\dB}(\mT, -): \D(\dB) \lto \D(\dA),
\end{align}
between the derived categories which are adjoint to each other.
Consider the category of all small DG $\kk$\!--linear categories. It is now known that it has a structure of cofibrantly
generated model category, in which weak equivalences are the
quasi-equivalences (see \cite{Ta}).
Thus, the localization of this category
 with respect to all quasi-equivalences has small
Hom-sets. This also gives that a morphism from $\dA$ to $\dB$
in the localization can be represented as $\dA\leftarrow \dA_{cof}\to\dB,$ where $\dA_{cof}\to \dA$ is a cofibrant replacement
 (see \cite{Ke2}).
 Now we give a definition of quasi-functors which one-to-one correspond to morphisms in the localization.

\begin{definition}
A DG $\dA\hy\dB$\!--bimodule $\mT$ is called a {\sf quasi-functor}
from the DG category $\dA$ to DG category $\dB$ if the tensor functor
$
(-)\stackrel{\bL}{\otimes}_\dA \mT: \D(\dA) \to \D(\dB)
$
sends any representable DG $\dA$\!--module to an object which is
isomorphic  to a representable $\dB$\!--module in $\D(\dB).$
\end{definition}

The category of quasi-representable DG
$\dB$\!--modules is equivalent to the homotopy category $\Ho(\dB).$ Hence, any quasi-functor
$\mT$ induces a functor $\Ho(\mT):\Ho(\dA)\to \Ho(\dB).$
This also imply that the restriction of the functor $(-)\stackrel{\bL}{\otimes}_\dA \mT$ on the
triangulated category $\prf\dA$ defines a functor $\prf\dA\to \prf\dB$ between categories of perfect DG modules.
If $\dA$ is a DG algebra, then the quasi-functor $\mT$ is perfect as $\dB$\!--module, because  as the DG $\dB$\!--module it is quasi-isomorphic to a representable
DG $\dB$\!--module.

Let us consider   the full subcategory $\Rep(\dA,\; \dB)\subset\D(\dA^{op}\otimes\dB)$  consisting of
all quasi-functors. It is now known  that the morphisms between DG categories  $\dA$ and $\dB$
in the localization of the category of all small DG $\kk$\!--linear
categories with respect to the
quasi-equivalences are in a natural bijection with the classes of isomorphisms
of $\Rep(\dA,\dB)$ (see \cite{To}).
By this reason any morphism $\dA\stackrel{\sim}{\leftarrow} \dA_{cof}\to\dB$ in the localization
will be called a quasi-functor.

\subsection{Triangulated categories and semi-exceptional collections}
Let $\T$ be a triangulated category. Consider a set of objects $S\subset\T.$ We say that the set $S$ {\sf classically generates}  the triangulated category $\T,$
if the smallest full triangulated subcategory of
$\T,$ which contains $S$ and is closed under direct summands, coincides with the whole category $\T.$
An object $E\in\T$ will be called a {\sf classical generator} of $\T,$ if the set $\{E\}$ classically generates $\T.$

A classical generator is called {\sf strong} if it generates the triangulated category $\T$ in a finite number of steps.
More precisely,
let $\I_1, \I_2 \subset
\mathcal{T}$ be two full subcategories.
Denote by $\I_1*\I_2$ the full subcategory of $\T$
consisting of all objects $X$ such that there is an exact triangle $X_1\to X\to X_2$ with $X_i\in \I_i.$
Let $\langle \I \rangle$ be the smallest full subcategory
containing $\I$ and
closed under shifts, finite direct sums, and direct summands.
Put
$\I_1 \diamond \I_2 := \langle \I_1 *
\I_2 \rangle,$
and define by induction
$\langle \I\rangle_k=\langle\I\rangle_{k-1}\diamond\langle \I\rangle.$ We denote $\langle \I\rangle$ by $\langle E\rangle_1,$ when $\I$ consists of a single object $E.$ 
By induction, we put $\langle E\rangle _k=\langle E\rangle_{k-1}\diamond\langle E\rangle_1.$
Finally, an object $E$ is called a {\sf strong generator} if we have $\langle E \rangle_n=\T$ for some $n\in\NN.$

\begin{definition}\label{regprop}
A triangulated category $\T$ is called  {\sf regular} if it has a strong generator.
\end{definition}

From now we will consider only  $\kk$\!--linear triangulated categories, where $\kk$ is a base field.
At first, we recall a notion of a proper triangulated category.
\begin{definition}\label{prop}
A $\kk$\!--linear triangulated category  $\T$ will be  called {\sf proper} if the vector spaces
$\bigoplus_{m\in\ZZ}\Hom(X, Y[m])$ are finite-dimensional for any pair of objects $X, Y\in\T.$
\end{definition}
 Regular and proper idempotent complete triangulated categories have many good properties.
 In particular, all of them will be admissible.

Let $\N\subset\T$ be a full triangulated subcategory.
It is called {\sf right admissible}
(resp. {\sf left admissible}) if the inclusion functor has a right
(resp. left) adjoint $\T\to \N.$ The
subcategory $\N\subset\T$ is called {\sf admissible}, if it is right
and left admissible at the same time.
Assume that $\T$ is a proper triangulated category and $\N\subset\T$ is regular and  idempotent complete, then
$\N$ will be admissible in $\T.$

\begin{theorem}{\rm \cite[Th. 1.3]{BV}}\label{saturated}
Let $\T$ be a regular and proper idempotent complete triangulated category.
Then any cohomological functor  from $\T^{\op}$ to the category of finite-dimensional
vector spaces is representable, i.e. it has the form $h^{Y}=\Hom(-, Y).$
\end{theorem}

In the paper \cite{BK} such  a triangulated category $\T,$ for which any cohomological functor  from $\T^{\op}$ to the category of finite-dimensional
vector spaces is representable, is called {\sf right saturated} (see \cite{BK, BV}).
It was proved in \cite{BK} that any right saturated  full triangulated subcategory
in  a proper triangulated category will be right admissible.
Since the opposite category is also proper and regular,  Theorem \ref{saturated} directly implies  the following statement.
\begin{proposition}\label{admissible}
Let $\N\subset\T$ be a full triangulated subcategory in  a proper triangulated category
$\T.$ Suppose that $\N$ is regular and  idempotent complete.  Then $\N$ is admissible in $\T.$
\end{proposition}

With any subcategory ${\N}\subset \T$ we can associate a full triangulated subcategory ${\N}^{\perp}\subset {\T}$ (resp. ${}^{\perp}{\N}$)
that is called a  {\sf right orthogonal} (resp. {\sf left
orthogonal}). By definition, the
subcategory ${\N}^{\perp}\subset {\T}$ (or ${}^{\perp}{\N}$) consists of all objects $M,$ for which
 ${\Hom(N, M)}=0$ (or ${\Hom(M, N)}=0$)  for all $N\in{\N}$
(see, e.g., \cite{BK, BO}).

\begin{definition}\label{sd}
A {\sf semi-orthogonal decomposition} of a triangulated category
$\T$ is a sequence of full triangulated subcategories ${\N}_1,
\dots, {\N}_n$ in ${\T}$ such that there is an increasing filtration
$0=\T_0\subset\T_1\subset\cdots\subset\T_n=\T$ by left admissible
subcategories for which the left orthogonals ${}^{\perp}\T_{r-1}$ in
$\T_{r}$ coincide with $\N_r$ for all $1\le r\le n.$
in this case  we write $ {\T}=\left\langle{\N}_1, \dots,
{\N}_n\right\rangle. $
\end{definition}

When the triangulated category $\T$ has a semi-orthogonal decomposition
${\T}=\left\langle{\N}_1, \dots, {\N}_n\right\rangle,$ where each
$\N_r$ is equivalent to $\prf\kk,$ we obtain a definition of a full exceptional collection.

\begin{definition}\label{exc}
An object $E$ of a $\kk$\!--linear triangulated category ${\T}$ is
called {\sf exceptional} if  ${\Hom}(E, E[m])=0$ for all $m\ne 0,$
and ${\Hom}(E, E)\cong\kk.$ An {\sf exceptional collection} in ${\T}$ is
a sequence of exceptional objects $\sigma=(E_1,\dots, E_n)$
satisfying the semi-orthogonality condition ${\Hom}(E_i, E_j[m])=0$
for all $m$ if $i>j.$ The collection $\sigma$ is called {\sf full}, if the set $\{E_1,\dots, E_n\}$ classically generates the whole category
$\T.$
\end{definition}

When the field $\kk$ is not algebraically closed it is reasonable to weaken the notions of an exceptional
object and an exceptional collection.
\begin{definition}\label{semi-exc}
An object $E$ of a $\kk$\!--linear triangulated category ${\T}$ is
called {\sf semi-exceptional} if ${\Hom}(E, E[m])=0$ for all $m\ne 0$ and
${\Hom}(E, E)=S,$ where $S$ is a semisimple finite-dimensional $\kk$\!--algebra.
A {\sf semi-exceptional  collection} in ${\T}$ is a
sequence of semi-exceptional objects $(E_1,\dots, E_n)$ with
semi-orthogonality conditions ${\Hom}(E_i, E_j[m])=0$ for all $m$ if
$i>j.$
\end{definition}
It is evident that semi-exceptional objects and collections are stable under base field change.
As above, a semi-exceptional collection
$\sigma=(E_1,\dots, E_n)$ is called {\sf full}, if it classically generates the whole category $\T.$  In this case the triangulated category $\T$ has a
semi-orthogonal decomposition with $\N_r=\langle E_r\rangle,$ and  we will use the following notation $ \T=\langle E_1,\dots, E_n \rangle.$

\begin{example}\label{Severi-Brauer}
{\rm
Let $\ff$ be a field and $D$ be a central simple algebra over $\ff.$
Consider the Severi-Brauer variety $SB(D).$ It is a smooth projective scheme over $\ff.$
The triangulated category $\prf SB(D)$  of perfect complexes on $SB(D)$ has a full semi-exceptional collection
$(S_0, S_1,\dots S_n)$   such that $S_0=\cO_{SB}$ and
$\End(S_i)\cong D^{\otimes i},$ where $n$ is the dimension of the scheme $SB(D).$
 Since each $D^{\otimes i}$ is a matrix algebra over a central division algebra $D_i,$ there is a semi-exceptional collection
$(E_0, E_1,\dots, E_n)$ such that $\End(E_i)\cong D_i$ (see \cite{Be} for a proof).
In this situation $S_i$ are isomorphic to $E_i^{\oplus k_i}$ for some integers $k_i.$
If now $\ff$ is a finite separable extension of $\kk,$ the projective scheme $SB(D)$ can be considered as a smooth projective
scheme over $\kk.$ The triangulated category $\prf SB(D)$ doesn't depend on the base field, but instead $\ff$\!--linear structure we now consider only $\kk$\!--linear structure. Note also that if $D$ is a field $\ff,$ then $SB(D)$ is 0-dimensional and it is  just $\Spec\ff.$
}
\end{example}

\subsection{Derived noncommutative schemes and their gluing}\label{gluDNS}

A {\sf derived noncommutative scheme} $\dX$ over a field $\kk$ is a $\kk$\!--linear DG category of the form $\prfdg\dR,$
where $\dR$ is a cohomologically  bounded DG $\kk$\!--algebra. Under such a definition the triangulated category $\prf\dR$
is called the category of perfect complexes on the scheme $\dX,$ while the derived category
$\D(\dR)$ will be called the derived category of quasi-coherent sheaves on it (see \cite[Def. 2.1]{O_alg}).

\begin{definition}\label{propdef}
A noncommutative scheme $\dX=\prfdg\dR$ will be called {\sf proper} if the triangulated category $\prf\dR$ is proper and it is called {\sf regular} if the triangulated category $\prf\dR$ is regular
\end{definition}
It is not difficult to show that the derived noncommutative scheme $\dX=\prfdg\dR$ is proper iff
the cohomology algebra $\bigoplus_{p\in\ZZ}H^p(\dR)$ is finite-dimensional. In particular, for any finite-dimensional DG algebra $\dR$ the
derived noncommutative scheme $\dX=\prfdg\dR$ is proper.

The regularity is directly related to a smoothness.
However, a property to be smooth depends on the base field $\kk.$

\begin{definition}
A noncommutative scheme $\dX=\prfdg\dR$ is called {\sf $\kk$\!--smooth} if the DG algebra $\dR$ is perfect as the DG bimodule, i.e. as the DG module over  $\dR^{\op}\otimes_{\kk}\dR.$
\end{definition}

In was proved in \cite{LS} that  smoothness is invariant under Morita equivalence.
In particular, the DG category $\prfdg\dR$ is smooth if and only if the DG algebra $\dR$ is smooth.
Any smooth DG category is also regular (see \cite{Lu}). Hence, any smooth noncommutative scheme will be regular too.

Let now $\dX=\prfdg\dR$ and $\dY=\prfdg\dS$ be two arbitrary derived
noncommutative schemes over a base field $\kk.$
A {\sf morphism} $\mf:\dX \to \dY$ between them  is a quasi-functor
$\mF:\prfdg\dS\to \prfdg\dR.$ Each morphism $\mf$ induces the derived functors
\[
\bL \mf^*:=\bL F^*: \D(\dS)\lto \D(\dR)
\quad
\text{and}
\quad
\bR\mf_*:=\bR F_*: \D(\dR)\lto \D(\dS)
\]
 which were defined for a quasi-functor $\mF$ in (\ref{derived_quasi}). The functor  $\bL \mf^*$ and $\bR\mf_*$ will be called the inverse image and the direct image functors, respectively.
As usually, the inverse image functor sends perfect modules to perfect ones, while  the direct image functor commutes with all direct sums.

Now we recall a definition of a gluing operation. Let $\dA$ and $\dB$ be small DG categories and let $\mT$ be some DG $\dB\hy\dA$\!--bimodule,
i.e. a DG  $\dB^{\op}\otimes\dA$\!--module.
One can define a so called lower triangular DG category that is related to  the data
$(\dA, \dB; \mT)$ (see \cite{Ta, KL, O_glue, O_Krull, O_alg}).

\begin{definition}\label{upper_tr}
Let $\dA$ and $\dB$ be two small DG categories and let $\mT$ be a DG $\dB\hy\dA$\!--bimodule.
The {\sf lower triangular} DG category $\dC=\dA\underset{\mT}{\with}\dB$ is defined as follows:
\begin{enumerate}
\item[1)] $\Ob(\dC)=\Ob(\dA)\bigsqcup\Ob(\dB),$

\item[2)]
$
\dHom_{\dC}(\mX, \mY)=
\begin{cases}
 \dHom_{\dA}(\mX, \mY), & \text{ when $\mX, \mY\in\dA$}\\
     \dHom_{\dB}(\mX, \mY), & \text{ when $\mX, \mY\in\dB$}\\
      \mT(\mY, \mX), & \text{ when $\mX\in\dA, \mY\in\dB$}\\
      0, & \text{ when $\mX\in\dB, \mY\in\dA$}
\end{cases}
$
\end{enumerate}
with the composition law coming from the DG categories $\dA, \dB$ and the bimodule structure on $\mT.$
\end{definition}

\begin{definition}\label{gluing_cat}
Let $\dX=\prfdg\dR$ and $\dY=\prfdg\dS$ be two derived noncommutative schemes and let $\mT$ be a DG $\dS\hy\dR$\!--bimodule. A {\sf gluing} $\dX\underset{\mT}{\oright}\dY$ of $\dX$ and $\dY$ via (or with respect to) $\mT$
is the DG category
$\prfdg\,(\dR\underset{\mT}{\with}\dS).$
\end{definition}
The DG category $\prfdg\,(\dR\underset{\mT}{\with}\dS)$ gives a derived noncommutative scheme
exactly in the case when the DG algebra $\dR\underset{\mT}{\with}\dS$ is cohomologically bounded.
It happens if and only if the  DG bimodule $\mT$ has only finitely many nontrivial cohomology.
We will write simply $\dX\oplus\dY,$ when $\mT\cong \m0.$ In this case, we have $\dX\oplus\dY\cong \dY\oplus\dX.$

It is almost evident that for any gluing $\prfdg\,(\dR\underset{\mT}{\with}\dS)$
the triangulated category $\prf\,(\dR\underset{\mT}{\with}\dS)$ has the following semi-orthogonal decomposition
$\langle\prf\dR, \prf\dS\rangle.$
It also can be shown that any semi-orthogonal decomposition is induced by a gluing. We have the following proposition.

\begin{proposition}\label{gluing_semi-orthogonal}\cite[Prop 3.8]{O_glue}
Let $\dE$ be a DG $\kk$\!--algebra. Suppose there is a semi-orthogonal decomposition
$\prf\dE=\langle \A, \B\rangle.$ Then the DG category $\prfdg\dE$ is quasi-equivalent to a gluing
$\dA\underset{\mT}{\oright}\dB,$ where $\dA, \dB\subset\prfdg\dE$ are full DG subcategories
with the same objects as $\A$ and $\B,$ respectively, and the DG $\dB\hy\dA$\!--bimodule is given by the rule
\begin{equation}\label{bimodule}
\mT ( \mY, \mX)=\dHom_{\dE}(\mX, \mY), \quad \text{with}\quad \mX\in\dA \;\text{and}\; \mY\in\dB.
\end{equation}
\end{proposition}

Note that the DG categories $\dA$ and $\dB$ also have the form $\prfdg\dR$ and $\prfdg\dS,$ respectively, and the DG category
$\prfdg\dE$ is quasi-equivalent to $\prfdg\,(\dR\underset{\mT}{\with}\dS).$

Derived noncommutative schemes, which are obtained by gluing, inherit many properties of its components if the gluing DG bimodule is also good enough.
For example, it is almost evident that the gluing
$\dX\underset{\mT}{\oright}\dY$ is proper if and only if
the components  $\dX, \dY$ are proper and   $\dim_{\kk}\bigoplus_i  H^i(\mT)<\infty .$
It is not difficult to show that the  gluing $\dX\underset{\mT}{\oright}\dY$ is regular if and only if
the components $\dX$ and $\dY$ are regular (see \cite[Prop. 3.20, 3.22]{O_glue}) and, hence,
this property does not depend on the DG bimodule $\mT.$
However, it is important to note that  the smoothness depends on $\mT.$
The following proposition is proved in \cite{LS}.

\begin{proposition}\cite[3.24]{LS}\label{smooth_glue}
Let $\dX=\prfdg\dR$ and $\dY=\prfdg\dS$ be two derived noncommutative schemes over $\kk.$
Then the following conditions are equivalent.
\begin{enumerate}
\item[1)] The gluing derived noncommutative schem $\dX\underset{\mT}{\oright}\dY$ is smooth,
\item[2)] The components $\dX$ and $\dY$ are smooth and the DG $\dS\hy\dR$\!--bimodule $\mT$ is a perfect.
\end{enumerate}
\end{proposition}

\section{Finite-dimensional DG algebras and Auslander construction }\label{applications}

\subsection{Finite-dimensional DG algebras}

Let $\dR=(\gR, \dr)$ be a finite-dimensional DG algebra over a base field $\kk.$
The algebra $R$ with the grading forgotten will be denoted by $\un{R}.$ In other words,  $\un{R}$ is the underlying ungraded algebra of $\gR.$
We consider DG algebras with identity $1\in R^0$ and $\dr(1)=0.$ Moreover, we will assume that the DG algebra $\dR$ is not acyclic, i.e. it is not quasi-isomorphic to $0.$
This means that the cohomology algebra $H^*(\dR)$ is not $0,$ and it is equivalent to ask that the image of $1$ in the cohomology $H^0(\dR)$ is not equal to $0.$

Denote by $\rd=\rd(\un{R})\subset\un{R}$ the (Jacobson) radical of the $\kk$\!--algebra $\un{R}.$ We know that $\rd^s=0$ for some $s.$
Define the index of nilpotency $n=n(\un{R})$ of $\un{R}$ as the smallest such integer $s$
that $\rd^s=0.$

Let us denote by $\Aut_{\kk}(\un{R})$ the group scheme of automorphisms of the ungraded $\kk$\!--algebra $\un{R}.$
Any $\ZZ$\!--grading on $\un{R}$ is the same as an action of the multiplicative group scheme ${\mathbf G}_{m,\kk}$ on $\un{R},$ i.e.
it is given by a homomorphism of the group schemes
${\mathbf G}_{m,\kk}\to \Aut_{\kk}(\un{R}).$
The radical $\rd$ is invariant with respect to the action of the group scheme $\Aut_{\kk}(\un{R})$ and, hence, it is also $\ZZ$\!--graded.
Thus, $\rd\subset\gR$ is a graded ideal and we denote by $\gS$ the quotient graded algebra $\gR/\rd.$
The underline ungraded algebra $\un{S}$ is semisimple.

\begin{remark}\label{remuse}{\rm
Let us note that the radical $J$ does not necessary contain the sum $R_{\ne 0}:=\bigoplus\limits_{q\ne 0} R^q,$
because $R_{\ne 0}$ is not necessary an ideal. On the other hand, any homogeneous element of $s\in R_{i}, i\ne 0$ is a nilpotent in the case
when the algebra $R$ is finite-dimensional over $\kk.$ Therefore, any $\ZZ$\!-grading on a product $D_1\times\cdots\times D_m$ of finite-dimensional division $\kk$\!--algebras $D_i$
is trivial.
}
\end{remark}

It is easy to see that  the radical $\rd\subset R$ is not necessary a DG ideal, in general, and $\dr(\rd)$ is not necessary a subspace of $\rd.$
In fact, with any two-sided graded ideal $I\subset\gR$ we can associate two DG ideals $\mI_{-}$ and $\mI_{+}.$

\begin{definition}\label{intext}
Let $\dR=(\gR, \dr)$ be a finite-dimensional DG algebra and $I\subset\gR$ be a graded (two-sided) ideal.
An {\sf internal} DG ideal $\mI_{-}=(I_{-}, \dr)$ consists of all $r\in I$ such that $\dr( r)\in I,$ while
an {\sf external} DG ideal  $\mI_{+}=(I_{+}, \dr)$  is the sum $I+\dr (I).$
\end{definition}

It is easy to see that $I_{-}$ and $I_{+}$ are really two-sided graded ideals of $\gR.$ It is evident that they are closed under the action of the differential $\dr.$
Thus, for any (two-sided) ideal $I\subset\gR$ we obtain two DG ideals  $\mI_{-}$ and $\mI_{+}$ in the DG algebra $\dR.$
If the ideal $I\subset\gR$ is closed under action of the differential $\dr,$ then the internal and external DG ideals $\mI_{-}$ and $\mI_{+}$
coincide with the ideal $I.$

\begin{definition}
Let $\rd\subset\gR$ be the radical of $\gR.$ The DG ideals $\rdi$ and $\rde$ will be called {\sf internal} and {\sf external} DG radicals of DG algebra $\dR.$
\end{definition}

Let $\gS_{-}$ and $\gS_{+}$ be the quotient graded algebras $\gR/\rd_{-}$ and $\gR/\rd_{+},$ respectively. The differential $\dr$ on $\gR$ induces differentials on these graded algebras
and we obtain two DG algebras $\dS_{-}=\dR/\rdi$ and $\dS_{+}=\dR/\rde.$

\begin{lemma}\label{quasi-is}
The natural homomorphism of DG algebras $\dS_{-}\to\dS_{+}$ is a quasi-isomorphism.
\end{lemma}
\begin{proof}
The morphism $\dS_{-}\to\dS_{+}$ is surjective and its kernel is the DG module $\rde/\rdi.$ Thus, we have to check that the complex $\rde/\rdi$ is acyclic.
Let $x\in\rde$ be an element such that  $\dr (x)=w$ with $w\in\rdi.$ We know that $x=y+\dr (z),$ where $y, z\in\rd.$
Since $\dr (y)=w,$ then, by definition $\rdi,$ we have $y\in\rdi.$ Therefore $\bar{x}=\dr (\bar{z})$ in the quotient  $\rde/\rdi,$
where $\bar{x}, \bar{z}$ are images of the elements $x, z$ in the quotient $\rde/\rdi.$
\end{proof}
This implies that a priori different  DG ideals $\rdi$ and $\rde$ give us quasi-isomorphic DG algebras $\dS_{-}$ and $\dS_{+}.$
A useful property of the DG ideal $\rdi$ is its nilpotency as a subideal of $\rd,$ while a useful property of the DG ideal $\rde$
is semisimplicity of the underline ungraded algebra $\un{S}_{+}$ as the quotient of the semisimple algebra $\un{S}=\un{R}/\rd.$

With any DG algebra $\dR=(\gR, \dr)$ we can associate the DG subalgebra $\dZ(\dR)=(Z(\dR), 0)$ with zero differential, which  consists of all cocycles, i.e. elements $r\in R$ with $d(r)=0.$ The DG algebra $\dZ(\dR)$ has a natural two-sided DG ideal $\dI(\dR)$ that is formed by all coboundaries. The quotient DG algebra $\dH(\dR)=\dZ(\dR)/\dI(\dR)$ is called the cohomology DG algebra for $\dR.$

The following proposition is a particular case of a more general statement \cite[Lemma 22.27.6]{StPr}, but we give a short proof here for the case of  finite-dimensional DG algebras.

\begin{proposition}\label{obj}
Let $\dR=(\gR, \dr)$ be a finite-dimensional DG algebra over a base field $\kk.$ Let $\mM\in \prfdg\dR$ be a perfect DG $\dR$\!--module.
Then $\mM$ is homotopy equivalent to a finite-dimensional h-projective DG $\dR$\!--module and  the DG endomorphism algebra $\dEnd_{\dR}(\mM)$ is quasi-isomorphic to a finite-dimensional DG algebra.
\end{proposition}
\begin{proof} If the DG module $\mM$ belongs to the full DG subcategory $\SFf\dA\subset \SF\dA$ of finitely generated semi-free
DG modules, then it is finite-dimensional and the DG algebra $\dEnd_{\dR}(\mM)$ is finite-dimensional too.
By definition of  perfect DG modules, it is isomorphic  to  a direct summand of an object from $\SFf\dA$
in the homotopy category $\Ho(\SF\dR)\cong\D(\dR).$ Let $\mN\in\SFf\dR$ be a such DG module that $\mM$ is a direct summand of
$\mN$ in $\prf\dR.$ Consider the finite-dimensional DG algebra $\dE=\dEnd_{\dR}(\mN).$ Since $\mM$ is a direct summand of $\mN$ in
$\prf\dR,$ there is an idempotent $e\in \End_{\prf\dR}(\mN)= H^0(\dEnd_{\dR}(\mN))$ which acts trivially on $\mM$ and acts as the identity on a complement to $\mM$
in $\mN.$ Consider the surjective map of the graded finite-dimensional algebras $Z(\dE)\to H(\dE).$
It is a well-known that under a surjective homomorphism of finite-dimensional algebras  any idempotent can be lifted.
Thus,  the idempotent $e\in H^0(\dE)$ can be lifted to an idempotent $\widetilde{e}\in Z^0(\dE).$ Now it can be considered as an element of $\dE.$ It induces an endomorphism $\widetilde{e}:\mN\to \mN$ of the DG
$\dR$\!--module $\mN.$
Denote by $\mM'$ the DG submodule of $\mN$ which is the kernel of this endomorphism. Since $\widetilde{e}$ is an idempotent, it gives a decomposition of DG modules $\mN=\mM'\oplus\mK.$
Thus, the DG module $\mM'$ is isomorphic to $\mM$ in the derived category $\D(\dR).$ Moreover, the DG module $\mM'$ is h-projective as a direct summand of a semi-free DG module $\mN.$ Hence, the h-projective DG modules $\mM$ and $\mM'$ are homotopy equivalent in the DG category of all h-projective DG
$\dR$\!--modules. Finally, we obtain that the DG algebra of endomorphisms $\dEnd_{\dR}(\mM)$ is quasi-isomorphic to the finite-dimensional DG algebra $\dEnd_{\dR}(\mM').$
\end{proof}

The following corollary is an immediate consequence of Proposition \ref{obj}.

\begin{corollary}\label{fin_diff}
Let $\dR_1$ and $\dR_2$ be two DG $\kk$\!--algebras. Assume that the DG categories $\prfdg\dR_1$ and $\prfdg\dR_2$ are quasi-equivalent.
If the DG algebra $\dR_1$ is quasi-isomorphic to a finite-dimensional DG algebra, then the DG algebra $\dR_2$ is also
quasi-isomorphic to some finite-dimensional DG algebra.
\end{corollary}

Let us now take two derived noncommutative schemes $\dX=\prfdg\dR$ and $\dS=\prfdg\dS$ such that the DG $\kk$\!--algebras $\dR$ and $\dS$ are finite-dimensional.
We can consider a gluing $\dZ=\dX\underset{\mT}{\oright}\dY$ with respect to a DG $\dS\hy\dR$\!--bimodule $\mT$ for which   $\dim_{\kk}\bigoplus_i  H^i(\mT)<\infty .$
In this case the new derived noncommutative scheme $\dZ$ is also proper. It has the form $\prfdg(\dR\underset{\mT}{\with}\dS),$ where $\dR\underset{\mT}{\with}\dS$
is a lower triangular DG algebra from Definition \ref{upper_tr}. It is important and useful do know when the lower triangular DG algebra $\dR\underset{\mT}{\with}\dS$
is quasi-isomorphic to a finite-dimensional DG algebra.

\begin{proposition}\label{glu_fs}
Let $\dR$ and $\dS$ be two finite-dimensional DG $\kk$\!--algebras and let $\mT$ be a perfect DG $\dS\hy\dR$\!--bimodule.
Then the lower triangular DG algebra $\dR\underset{\mT}{\with}\dS$ is quasi-isomorphic to a finite-dimensional DG $\kk$\!--algebra.
\end{proposition}
\begin{proof}
Consider the tensor product $\dS^{\op}\otimes_{\kk}\dR.$ It is a finite-dimensional DG $\kk$\!--algebra and, by assumption, we have $\mT\in\prfdg(\dS^{\op}\otimes_{\kk}\dR).$
Now, by Proposition \ref{obj}, the DG bimodule $\mT$ is quasi-isomorphic to a finite-dimensional DG bimodule. This implies that the lower triangular DG algebra $\dR\underset{\mT}{\with}\dS$ is quasi-isomorphic to a finite-dimensional DG algebra.
\end{proof}
\begin{proposition}\label{glu_fin}
Let $\dR$ and $\dS$ be two finite-dimensional DG $\kk$\!--algebras and let $\mT$ be a DG $\dS\hy\dR$\!--bimodule such that $\dim_{\kk}\bigoplus_i  H^i(\mT)<\infty .$
If the DG algebras $\dR$ and $\dS$ are $\kk$\!--smooth then the lower triangular DG algebra $\dR\underset{\mT}{\with}\dS$ is also smooth and it is quasi-isomorphic to a finite-dimensional DG algebra.
\end{proposition}
\begin{proof}
If the DG algebras $\dR$ and $\dS$ are $\kk$\!--smooth then the tensor product $\dS^{\op}\otimes_{\kk}\dR$ is also $\kk$\!--smooth. Therefore, the triangulated category
$\prf(\dS^{\op}\otimes_{\kk}\dR)$ is proper and regular. The object $\mT$ gives a cohomological functor $\Hom(-, \mT)$ to the category of
finite-dimensional vector spaces. By Theorem \ref{saturated}, it is representable. Hence, $\mT$ is quasi-isomorphic to a perfect DG bimodule.
By Proposition \ref{smooth_glue} the gluing $\prfdg\dR\underset{\mT}{\oright}\prfdg\dS$ is smooth and, hence, the DG algebra $\dR\underset{\mT}{\with}\dS$ is also smooth.
By Proposition \ref{glu_fs}, we obtain  that the lower triangular DG algebra $\dR\underset{\mT}{\with}\dS$ is quasi-isomorphic to a finite-dimensional DG algebra.
\end{proof}
\begin{corollary}\label{sep_coll}
Let $\dR$ be a DG $\kk$\!--algebra. Assume that the triangulated category $\prf\dR$ is proper and it has a full separable semi-exceptional collection.
Then the DG algebra $\dR$ is $\kk$\!--smooth and it is quasi-isomorphic to a finite-dimensional DG $\kk$\!--algebra.
\end{corollary}
\begin{proof}
Consider the objects $\mE_1,\ldots, \mE_k$ in the DG category $\prfdg\dR$ which form a full separable semi-exceptional collection
in the triangulated category $\prfdg\dR.$ Denote by $\dE$ the DG  emdomorphism algebra $\dEnd_{\prfdg\dR}(\mE),$ where
$\mE=\bigoplus_{i=1}^{k}\mE_i.$ The object $\mE$ is a classical generator for $\prf\dR.$ Therefore
the DG category $\prfdg\dR$ is quasi-equivalent to $\prfdg\dE$ (see, e.g., \cite{Ke, Ke2}, \cite[Prop. 1.15]{LO}).
Let us prove induction by $k$ that the DG algebra $\dE$ is smooth and is quasi-isomorphic to a finite-dimensional DG $\kk$\!--algebra.
The base of induction is a separable semisimple finite-dimensional $\kk$\!--algebra that is smooth.
By Proposition \ref{gluing_semi-orthogonal} the DG category $\prfdg\dE$ is  quasi-equivalent to a gluing $\prfdg\dE'\underset{\mT}{\oright}\prfdg\dE_k,$ where
$\dE'=\dEnd_{\prfdg\dR}(\bigoplus_{i=1}^{k-1}\mE_i)$ and $\dE_k=\dEnd_{\prfdg\dR}(\mE_k).$ By induction hypothesis, we can assume that
the DG algebras $\dE'$ and $\dE_k$ are finite-dimensional and smooth. Since $\prf\dR$ is proper, the DG $\dE_k\hy\dE'$\!--bimodule
$\mT$ is cohomologicaly finitely-dimensional, i.e.  $\dim_{\kk}\bigoplus_i  H^i(\mT)<\infty .$ Now by Proposition \ref{glu_fin}
the DG algebra $\dE$ is smooth and it is quasi-isomorphic to a finite-dimensional DG algebra.
By Corollary \ref{fin_diff} the DG algebra is also quasi-isomorphic to a finite-dimensional DG algebra. It is also smooth, because
the smoothness is a property of the DG category $\prfdg\dR\cong\prfdg\dE.$
\end{proof}
\begin{remark}
{\rm
Note that there is an example of a gluing of non-smooth finite-dimensional DG algebras via a cohomologically finite DG bimodule such that
the resulting lower triangular DG algebra is not quasi-isomorphic to a finite-dimensional DG algebra. Such example can be found in Alexander Efimov's paper \cite{Ef}
in Theorem 5.4. It is proved that the DG algebra constructed there (actually $A_{\infty}$\!--algebra) does not have a categorical resolution of singularities, while any finite-dimensional DG algebra has such a resolution
as it will be shown in Theorem \ref{DGalgebra}.
}
\end{remark}

\subsection{Semisimple finite-dimensional DG algebras}
In this section we introduce and compare two  different notions of simplicity and semisimplicity
for finite-dimensional DG algebras.
Let $\dS=(\gS, \ds)$ be a finite-dimensional DG algebra over a base field $\kk.$
\begin{definition}\label{simple}
A finite-dimensional DG algebra $\dS$ is called {\sf simple},
if the DG category $\prfdg\dS$ of perfect DG modules is quasi-equivalent to $\prfdg D,$ where $D$ is a finite-dimensional division
$\kk$\!--algebra. It is called {\sf semisimple}, if  the DG category $\prfdg\dS$ is quasi-equivalent to a gluing $\prfdg D_1\oplus\cdots\oplus \prfdg D_m,$ where all $D_i$ are finite-dimensional division algebras over $\kk.$
In addition, $\prfdg\dS$ is called {\sf separable}, if all division algebras $D_i$ are separable over $\kk.$
\end{definition}
\begin{remark}{\rm
Recall that a division $\kk$\!--algebra $D$  is called {\sf separable} over $\kk$ if it is a projective as
$D^{\op}\otimes_{\kk} D$\!--module. It is well-known that a division algebra $D$ is separable iff
its center is a separable extension of $\kk.$
}
\end{remark}

It follows directly from the definition that the notions of simple and semisimple DG algebras are invariant with respect to quasi-isomorphisms.
\begin{remark}\label{semi-ex}
{\rm
Note that the DG category of perfect modules $\prfdg\dS$ for a semisimple DG algebra $\dS$ is generated by a single semi-exceptional object.
Indeed, the object $D_1\oplus\cdots\oplus D_m$ is semi-exceptional and classically generates the whole category $\prf\dS.$
}
\end{remark}

Now we consider another notion of (semi)simplicity.

\begin{definition}\label{abstrsimple}
A finite-dimensional DG algebra $\dS=(S, \ds)$ is called {\sf abstractly simple} (or {\sf abstractly semisimple}), if
the underlying ungraded algebra $\un{S}$ is simple (or semisimple).
\end{definition}
These notions are not invariant with respect to quasi-isomorphisms of DG algebras.
It is easy to see that any simple (or semisimple) DG algebra is quasi-isomorphic to abstractly simple (or semisimple) DG algebra.
Indeed, for any complex $\mN\in\prfdg D$ the DG algebra of endomorphisms $\dEnd_{D}(\mN)$ is abstractly simple.
If now the DG category $\prfdg\dS$ is quasi-equivalent to $\prfdg D,$ then the DG algebra $\dS$ is quasi-isomorphic to
$\dEnd_{D}(\mN)$ for some complex $\mN\in\prfdg D.$

In fact, we can also show that any abstractly simple algebra has the form $\dEnd_{D}(\mN)$ for some complex $\mN\in\prfdg D$ and, hence, it is simple.

\begin{proposition}\label{eqsimple}
Let $\dS=(\gS, \ds)$ be an abstractly simple DG algebra over a field $\kk.$
Then it is simple and it is isomorphic to $\dEnd_{D}(\mN)$ for some bounded complex $\mN$ of finite (right) $D$\!--modules.
\end{proposition}
\begin{proof}
Let $\dS=(\gS, \ds)$ be an abstractly simple DG algebra over a field $\kk.$ The ungraded algebra $\un{S}$ is simple and, hence, it is isomorphic to
a matrix algebra ${\mathbf M}_n(D)$ over a division algebra $D.$ Thus, $\un{S}=\End_{D}(\un{N}),$ where $\un{N}=D^{\oplus n}$ is a free right $D$\!--module of rank $n.$
Denote by $\ff\subset D$ the center of the division algebra $D.$ The field $\ff$ is a finite extension of $\kk.$

The  $\ZZ$\!--grading on the $\kk$\!--algebra  $S$ is given by a homomorphism of the group schemes
${\mathbf G}_{m,\kk}\to \Aut_{\kk}(\un{S}),$ where $\Aut_{\kk}(\un{S})$ the automorphism group scheme of the ungraded $\kk$\!--algebra $\un{S}.$
The center of the algebra $\un{S}$ is the field $\ff,$ which is the center of $D.$
Any $\ZZ$\!--grading on an algebra induces a $\ZZ$\!--grading on the center of this algebra. In our case,
the center is the field $\ff,$ which is finite extension of $\kk.$ This implies that  any $\ZZ$\!--grading on $\kk$\!-algebra $\ff$  is trivial (see Remark \ref{remuse}). Thus, the grading on $S$ is also a grading on this algebra considered  as $\ff$\!--algebra.
It is given by a homomorphism ${\mathbf G}_{m,\ff}\to \Aut_{\ff}(\un{S}),$ where $\Aut_{\ff}(\un{S})$ the group scheme over $\ff$ of automorphisms of the ungraded $\ff$\!--algebra $\un{S}.$

Let us consider the group scheme homomorphism $\Int:
{\mathbf G}{\mathbf L}_{1, \ff}(\un{S})\to \Aut_{\ff}(\un{S}),$
where ${\mathbf G}{\mathbf L}_{1, \ff}(\un{S})$ is the group scheme of units of the $\ff$\!-algebra $\un{S}$
and $\Int(s)$ the inner automorphism of $\un{S}$ induced by $s.$ Since $\un{S}$ is simple, by Skolem--Noether theorem, the homomorphism
$\Int$ is surjective and there is the following exact sequence of $\ff$\!--group schemes.
\[
1\lto {\mathbf G}_{m, \ff}\lto {\mathbf G}{\mathbf L}_{1, \ff}(\un{S})\lto \Aut_{\ff}(\un{S})\lto 1
\]
It is well-known that $\Ext^1({\mathbf G}_{m, \ff}, {\mathbf G}_{m, \ff})=0.$ Hence the homomorphism
${\mathbf G}_{m,\ff}\to \Aut_{\ff}(\un{S})$ can be lifted to a homomorphism
$\rho: {\mathbf G}_{m,\ff}\to {\mathbf G}{\mathbf L}_{1, \ff}(\un{S}).$ Since the $\ff$\!--group scheme ${\mathbf G}{\mathbf L}_{1, \ff}(\un{S})$ is isomorphic to the $\ff$\!--group scheme $\Aut_{D}(\un{N}),$ we obtain a $\ZZ$\!--grading on the
$D$\!--module $\un{N}.$ Moreover, this grading on $D$\!--module $\un{N}$ induces the grading on $S.$ This means that  the graded algebra $S$ is isomorphic to the algebra $\End^{gr}_D(N)$ of graded endomorphisms,
where $N$ is the correspondent graded $D$\!--module.

Denote by $\dS'=(S, 0)$ and $\mN'=(N, 0)$ the DG algebra and the DG $D$\!--module with zero differentials. There is an isomorphism of between these DG algebras  $\dS'\cong \dEnd_{D}(\mN').$
Hence, DG algebra $\dS'$ is Morita equivalent to $D,$ i.e. the DG category $\prfdg\dS'$ is quasi-equivalent to the DG category $\prfdg D.$
The differential $\ds$ on the DG algebra $\dS$ gives a Hochschild 2-cochain for the DG algebra $\dS'$ that is actually a 2-cocycle, because the Leibnitz rule holds.
On the other hand, we know that the second Hochschild cohomology $HH^2(\dS')\cong HH^2(\prfdg\dS')\cong HH^2(D)$ are trivial.
Hence, any 2-cocycle is a 2-coboundary. Taking in account that the differential of $\dS'$ is trivial we obtain that there is
an element $d\in S^1$ such that
\[
\ds(s)=[d, s]:=d s-(-1)^q sd,\quad\text{for any}\quad s\in S^q.
\]
The property $\ds^2=0$ implies that the element $d^2$ belong to the center of the algebra $S.$ But the center of $S,$ which is the field $\ff,$ is in the degree $0.$
Therefore, we have $d^2=0.$ The element $d\in S^1=\End_{D}(N)^1$ defines a differential $d: N\to N.$ Denote by $\mN=(N, d)$ the complex of $D$\!--modules.
Now by construction of $\mN$ we see that the DG algebra $\dEnd_{D}(\mN)$ is isomorphic to $\dS.$
Since, by assumption, the DG algebra $\dS=\dEnd_{D}(\mN)$ is not quasi-isomorphic to the zero algebra, the  complex $\mN$ is not acyclic. Therefore, the DG algebra $\dS$ is Morita equivalent to the division algebra $D,$ i.e. $\prfdg\dS$ is quasi-equivalent to $\prfdg D.$ Thus, the DG algebra $\dS$ is simple.

We also can give a precise formula for the element $d\in S^1$ (if $\ff$ is not the field of two elements). Consider again the homomorphism $\rho: {\mathbf G}_{m,\ff}\to {\mathbf G}{\mathbf L}_{1, \ff}(\un{S})$ and denote by $\rho_{\lambda}\in S$ the image of an element $\lambda\in\ff^*$ under the map $\rho$ on $\ff$\!--points.
It is evident that $\rho_{\lambda}\in S^0$ and $\rho_{\lambda}s\rho_{\lambda}^{-1}=\lambda^q s$ for any $s\in S^q.$
Assume that $\lambda\ne 1$ and consider the following element
\[
d=\frac{\lambda}{1- \lambda}\rho_{\lambda}^{-1} \ds (\rho_{\lambda})\in S^1.
\]
Now taking in account the relation $\ds (\rho_{\lambda}^{-1})\rho_{\lambda}+\rho_{\lambda}^{-1} \ds (\rho_{\lambda})=0,$ we obtain the following sequence of equalities
\begin{multline*}
(\lambda^{-1} - 1)\ds (s)=
\rho_{\lambda}^{-1}\ds(\lambda^q s)\rho_{\lambda} - \ds (s)=
\rho_{\lambda}^{-1}(\ds(\rho_{\lambda}s\rho_{\lambda}^{-1})- \rho_{\lambda} \ds(s)\rho_{\lambda}^{-1})\rho_{\lambda}=\\
\rho_{\lambda}^{-1}\ds (\rho_{\lambda}) s + (-1)^q s \ds (\rho_{\lambda}^{-1})\rho_{\lambda} =
\rho_{\lambda}^{-1} \ds (\rho_{\lambda}) s - (-1)^q s \rho_{\lambda}^{-1} \ds (\rho_{\lambda})=
[\rho_{\lambda}^{-1} \ds (\rho_{\lambda}), s].
\end{multline*}
This shows that $[d,s]=\ds (s)$ for any element $s\in S^q.$
\end{proof}

This result can be easily extended to abstractly  semisimple DG algebras.
\begin{proposition}\label{eqsemisimple}
Let $\dS=(\gS, \ds)$ be an abstractly semisimple DG algebra over a field $\kk.$ Then it is semisimple and it is isomorphic to a product  $\dS_1\times\cdots\times \dS_n,$ where
all $\dS_i, 1\le i\le n,$ are abstractly simple DG $\kk$\!--algebras.
\end{proposition}
\begin{proof} Consider the decomposition $\un{S}=\un{S}_1\times \cdots\times \un{S}_n$ of the semisimple algebra $\un{S}$ as a direct product of simple algebras.
Let $e_1,\ldots, e_n$ be the complete set of central orthogonal idempotents such that $\un{S}_i=e_i \un{S}=\un{S}e_i=e_i\un{S}e_i.$
By Remark \ref{remuse} all central idempotents are homogeneous elements of degree $0.$ Therefore, all algebras $\un{S}_i$ are graded and
$S=S_1\times \cdots\times S_n$ as the graded algebra.
Now we have
\[
\ds (e_i)= \ds(e_i^2)=\ds (e_i) e_i +e_i \ds( e_i)=2 e_i \ds (e_i).
\]
Thus, we obtain that $\ds (e_i)\in S_i$ and, hence, $e_i\ds (e_i)=\ds(e_i).$ Furthermore, we have $e_i \ds (e_i)= 2 e_i \ds (e_i).$ This implies that $\ds (e_i)=0.$
Therefore, the differential $\ds$ sends any $S_i$ to itself and $\ds (e_i)=0$ for any $i.$ Hence,  the DG algebra $\dS$ is isomorphic to a product $\dS_1\times\cdots\times \dS_n.$
All DG algebras $\dS_i$ are abstractly simple and, by Proposition \ref{eqsimple}, they are also simple, when they are not acyclic.
This implies that $\dS$ is semisimple. We should to note that some DG algebras $\dS_i$ can be acyclic and in this case  $e_i=\ds (a_i)$ for some element $a_i.$
In such situation the DG category $\prfdg \dS_i$ is trivial.
\end{proof}
\begin{corollary}\label{intsimple}
Let $\dS=(\gS, \ds)$ be a finite-dimensional DG algebra over a base field $\kk.$ Assume that the internal radical $\rdi$ is trivial.
Then the DG algebra $\dS$ is semisimple.
\end{corollary}
\begin{proof} If $\rdi$ is trivial, then by Lemma \ref{quasi-is} the canonical morphism $\dS\to \dS/\rde$ is a quasi-isomorphism.
The DG algebra $\dS/\rde$ is abstractly semisimple, because the underlying ungraded algebra is a quotient of the semisimple algebra $S/J,$ where $J$ is the Jacobson radical of $S.$
By Proposition \ref{eqsemisimple}, the DG algebra $ \dS/\rde$ is semisimple. Hence, $\dS$ is also semisimple as a DG algebra that is quasi-isomorphic to a semisimple DG algebra.
\end{proof}

\subsection{Auslander construction for DG algebras}\label{ausl}

Let $\dR=(\gR, \dr)$ be a finite-dimensional DG algebra over a base field $\kk$ and let $\rd\subset\gR$ be the Jacobson radical of the algebra $\gR.$ As it was explained above, $\rd$ is a graded ideal and for any graded ideal we can define internal and external DG ideals (see Definition \ref{intext}). In particular, we have defined internal and external DG radicals $\rdi$ and $\rde.$

Consider the graded ideals $J\supset J^2\supset\cdots\supset J^n=0,$ where $n$ is the index of nilpotency of $\gR.$
Denote by $\mJ_p$ the internal DG ideals $((\rd^p)_{-}, \dr)$ for $p=1,\ldots, n.$
Thus, we obtain a chain of DG ideals  $\rdi=\mJ_{1}\supseteq\cdots\supseteq\mJ_n=0.$
Let us consider the following (right) DG modules $\mM_p=\dR/\mJ_p$ with $p=1,\ldots, n.$
In particular, we have $\mM_1\cong \dS_{-}$ and $\mM_n\cong\dR$ as the right DG $\dR$\!--modules.

Consider
the DG category $\Mod \dR$ of all right DG $\dR$\!--modules.
Denote by
$\dE=(\gE, d_{\dE})$  the DG algebra of endomorphisms $\dEnd_{\dR}(\mM)$ of the DG $\dR$\!--module $\mM=\bigoplus_{p=1}^{n} \mM_p$ in the DG category
$\Mod \dR.$ In particular, the graded algebra $\gE$ is the algebra of endomorphisms $\End^{gr}_{\gR}(M)$ as in \eqref{grHom} and the differential $d_{\dE}$ is defined by the formula
\eqref{DHom}.
Denote by $\mP_s$ the right DG $\dE$\!--modules $\dHom_{\dR}(\mM, \mM_s).$
All of them are h-projective DG $\dE$\!--modules and $\dE\cong\bigoplus_{s=1}^n \mP_s$ as the right DG $\dE$\!--module.
There are isomorphisms of complexes of vector spaces
\begin{align}\label{homall}
\dHom_{\dE}(\mP_i, \mP_j)\cong \dHom_{\dR}(\mM_i, \mM_j)=\dHom_{\dR}(\dR/\mJ_i, \dR/\mJ_j)
\quad\text{for all}
\quad
1\le i, j\le n.
\end{align}
Moreover, in the case when $ i\ge j$ we have isomorphisms
\begin{align}\label{homP}
\dHom_{\dE}(\mP_i, \mP_j)\cong\dHom_{\dR}(\dR/\mJ_i, \dR/\mJ_j)\cong \dR/\mJ_j.
\end{align}
Thus, for $i\ge j$ the canonical quotient morphisms $\dR/\mJ_i\to \dR/\mJ_j$  induce natural morphisms
$\phi_{i,j}: \mP_i\to \mP_j$ of the right DG $\dE$\!-modules.

Now we introduce into consideration a new sequence of DG $\dE$\!--modules $\mK_1,\ldots, \mK_n,$ where $\mK_1\cong\mP_1$ and
for any $1<i\le n$ we denote by
$\mK_i$ the cone $\Cone(\phi_{i, i-1})[-1]$ of the morphism $\phi_{i, i-1}$ shifted by $[-1].$
Thus, the DG $\dE$\!--module $\mK_i$ is isomorphic to the sum $\mP_{i-1}[-1]\oplus\mP_{i}$ as the graded $\gE$\!--module but with a differential of the form
\[
d_{\mK_i}=
\begin{pmatrix}
-d_{\mP_{i-1}}& -\phi_{i, i-1}\\
0 & d_{\mP_i}
\end{pmatrix}.
\]

A passing from a finite-dimensional DG algebra $\dR$ to the DG algebra $\dE$ can be considered as a generalization of the famous Auslander construction for finite-dimensional algebras (see \cite{Au}), and the  following theorem allows us to describe the category $\prf\dE.$
\begin{theorem}\label{maintech}
The defined above DG $\dE$\!--modules $\mK_1,\ldots, \mK_n$ satisfy the following properties:
\begin{enumerate}
\item[1)] For any $1\le i\le n$ the DG $\dE$\!--module $\mK_i$  is  h-projective.
\item[2)] The set DG modules $\{\mK_1,\ldots, \mK_n\}$ classically generate the triangulated category $\prf\dE.$
\item[3)] For any pair $i>j$ the complex $\dHom_{\dE}(\mK_i, \mK_j)$ is acyclic.
\item[4)] For any $1\le i\le n$ the DG endomorphism algebra $\dEnd_{\dE}(\mK_i)$ is semisimple.
\item[5)] If the DG algebra $\dR/\rde$ is separable, then DG algebras $\dEnd_{\dE}(\mK_i)$ are separable too.
\end{enumerate}
\end{theorem}
\begin{proof}
1) The DG $\dE$\!--modules $\mP_i$ are h-projective as direct summand of $\dE.$ Hence, the DG $\dE$\!--modules
$\mK_i=\Cone(\phi_{i, i-1})[-1]$ are also h-projective as cones in homotopy category $\Ho(\Mod\dE)$ of morphisms between h-projective DG modules.

2) Since $\dE\cong\bigoplus_{s=1}^n \mP_s,$ the DG $\dE$\!--modules $\mP_i, 1\le i\le n$ classically generate the whole
triangulated category $\prf\dE.$ Now $\mP_1\cong\mK_1$ and any $\mP_i$ is isomorphic to a cone of map from $\mP_{i-1}[-1]$ to $\mK_i$ in the triangulated category
$\prf\dE.$ Therefore, the DG modules $\mK_1,\ldots, \mK_n$ also classically generate the triangulated category $\prf\dE.$

3) Let us consider the complex $\dHom_{\dE}(\mK_i, \mP_j)$ with $i>j.$ This complex is isomorphic to the cone of a natural map
from $\dHom_{\dE}(\mP_{i-1}, \mP_j)$ to $\dHom_{\dE}(\mP_{i}, \mP_j),$ which is induced by $\phi_{i, i-1}.$ By \eqref{homP} both these complexes are isomorphic to $\dR/\mJ_j$ and the induced map
is the identity. Hence $\dHom_{\dE}(\mK_i, \mP_j)$ is acyclic, when $i>j.$ Now the complex $\dHom_{\dE}(\mK_i, \mK_j)$ is isomorphic to the cone of a natural map
from $\dHom_{\dE}(\mK_{i}, \mP_j)[-1]$ to $\dHom_{\dE}(\mK_{i}, \mP_{j-1})[-1].$ Thus, for $i>j$ the complex $\dHom_{\dE}(\mK_i, \mK_j)$ is acyclic as a cone of a map between acyclic complexes.

4) Consider the complex $\dHom_{\dE}(\mP_{i-1}, \mP_i)$ and denote in by $\mL_i.$ There is a natural embedding of the complex $\mL_i$ to $\dHom_{\dE}(\mP_{i}, \mP_i)\cong \dR/\mJ_i.$
Moreover, it is easy to see that the complex $\mL_i\subset\dR/\mJ_i$ is a two-sided DG ideal in the DG algebra $\dR/\mJ_i.$ Now consider the DG algebra $\dEnd_{\dE}(\mK_i).$ As a graded algebra it is isomorphic to
\[
\begin{pmatrix}
\dEnd_{\dE}(\mP_{i-1})& \dHom_{\dE}(\mP_{i}, \mP_{i-1}[-1])\\
\dHom_{\dE}(\mP_{i-1}[-1], \mP_i) & \dEnd_{\dE}(\mP_{i})
\end{pmatrix}
\cong
\begin{pmatrix}
\dR/\mJ_{i-1} & \dR/\mJ_{i-1}[-1]\\
\mL_i[1] & \dR/\mJ_i
\end{pmatrix}.
\]

However, the differential $D_i$ of the DG algebra $\dEnd_{\dE}(\mK_i)$ acts follows
\[
D_i
\begin{pmatrix}
 y & y'\\
l & x
\end{pmatrix}
=
\begin{pmatrix}
-d_{i-1} (y) -\bar{l} & -d_{i-1}(y')-\bar{x}+(-1)^{q}y\\
d_i (l) & d_i (x) +(-1)^{q} l
\end{pmatrix},
\]
where $y\in\dR/\mJ_{i-1}, x\in\dR/\mJ_i$ are element of degree $q,$ the element $y'\in\dR/\mJ_{i-1}$ has degree $q-1,$ and  $l\in \mL_{i}\subset\dR/\mJ_i$ is an element of degree $q+1.$
The elements $\bar{x}, \bar{l}$ are the images of $x, l$
under natural projection $\dR/\mJ_i\to \dR/\mJ_{i-1},$  while $d_i, d_{i-1}$ are the differentials on $\dR/\mJ_i$ and $\dR/\mJ_{i-1},$ respectively. The DG algebra
$\dEnd_{\dE}(\mK_i)$ has a two-sided ideal $\widetilde{\mL}_i$ of the form

\[
\widetilde{\mL}_i=
\begin{pmatrix}
\dR/\mJ_{i-1} & \dR/\mJ_{i-1}[-1]\\
\mL_i[1] & \mL_i
\end{pmatrix}
\]
that is closed under action of the differential $D_i.$ Therefore, $\widetilde{\mL}_i\subset\dEnd_{\dE}(\mK_i)$ is a DG ideal.
Moreover, it is acyclic and it is even homotopic to zero. Thus,  the natural projection
$\dEnd_{\dE}(\mK_i)\to \dEnd_{\dE}(\mK_i)/\widetilde{\mL}_i$ is quasi-isomorphism of DG algebras and the DG algebra $\dEnd_{\dE}(\mK_i)/\widetilde{\mL}_i$
is isomorphic to the DG algebra $(\dR/\mJ_i)/\mL_i.$

Now let us explore the DG algebra $(\dR/\mJ_i)/\mL_i$ and show that it is semisimple. Denote by $\widebar{\mL}_i\subset\dR$ the DG ideal that is the preimage of
$\mL_i$ with respect to the projection $\dR\to \dR/\mJ_i.$ Thus, we have $(\dR/\mJ_i)/\mL_i\cong \dR/\widebar{\mL}_i.$
It directly follows from the definition of $\mL_i\subset \dR/\mJ_{i}$ that the DG ideal $\widebar{\mL}_i$ consists of all elements $r\in \dR$ for which
$r\mJ_{i-1}\subseteq\mJ_i.$

Let us consider the DG ideal $\widebar{\mL}_i=(\widebar{L}_i, \dr)$ and the DG algebra $\dR/\widebar{\mL}_i.$ We would like to show that the internal DG radical of the DG algebra
$\dR/\widebar{\mL}_i$ is trivial. Take the Jacobson radical of the algebra $R/\widebar{L}_i$ and consider its preimage $\widebar{L}_i+J$ under the projection
$R\to R/\widebar{L}_i.$
Let $r$ be an element from the ideal  $\widebar{L}_i+J.$ Then, at first, we have $r(J^{i-1})_{-}\subseteq J^{i}.$ Now assume that
$\dr (r)$ belongs to $\widebar{L}_i+J.$
Consider any element $j\in (J^{i-1})_{-}$ and take $rj\in J^i.$ Let us show that $rj$ also belongs to the internal DG ideal $ (J^{i})_{-}.$
Indeed we have
\[
\dr(rj)=\dr( r) j +(-1)^{q} r \dr( j),\quad\text{where}\;\; q\;\;\text{is the degree of}\;\; r.
\]
Since $r, \dr( r)\in \widebar{L}_i+J$ and $j, \dr( j)\in (J^{i-1})_{-},$ we obtain that the element $\dr(rj)\in J^{i}.$ Hence, $rj\in J^i$ belongs to the internal ideal $(J^{i})_{-}$ in fact.
Finally, this implies that the element $r\in \widebar{L}_i+J$ belongs to $\widebar{L}_i$ in reality. This argument shows that the internal DG radical of the DG algebra
$\dR/\widebar{\mL}_i$ is trivial. Therefore, by Corollary \ref{intsimple}, the DG algebra $\dR/\widebar{\mL}_i$ is semisimple and the DG algebra $\dEnd_{\dE}(\mK_i),$ which is quasi-isomorphic to
$\dR/\widebar{\mL}_i,$ is also semisimple.

5) Since the DG algebra $\dS_{+}=\dR/\rde$ is semisimple, the DG category $\prfdg\dS_{+}$ is quasi-equivalent to a gluing $\prfdg D_1\oplus\cdots\oplus \prfdg D_m,$ where all $D_i$ are finite-dimensional division algebras over $\kk.$ Moreover, all of them are separable. The semisimple DG algebra $\dEnd_{\dE}(\mK_i)$ is quasi-isomorphic to
$\dR/\widebar{\mL}_i.$ Therefore, any summand in the decomposition of $\prfdg\dEnd_{\dE}(\mK_i)$ is quasi-equivalent to one of $\prfdg D_i.$
Thus, all semisimple DG algebras $\dEnd_{\dE}(\mK_i)$ are also separable over the base filed $\kk.$
\end{proof}

Let $\dR=(R, \dr)$ be a finite-dimensional DG algebra. With any such DG algebra we associated a DG algebra $\dE=(E, d_{\dE}),$  where $\dE=\dEnd_{\dR}(\mM)$
and $\mM=\bigoplus_{p=1}^{n} \mM_p=\bigoplus_{p=1}^{n}\dR/\mJ_p.$ This construction can be considered as a generalization of the Auslander construction for finite-dimensional algebras \cite{Au} to the case
of DG algebras.

Consider the DG $\dE$\!--module $\mP_n=\dHom_{\dR}(\mM, \mM_n).$ It is h-projective and  we have $\dEnd_{\dE}(\mP_n)\cong\dR$ by \eqref{homP}.
Thus, the DG $\dE$\!--modules $\mP_n$ is actually a DG $\dR\hy\dE$\!--bimodule and it induces two functors
\[
(-)\stackrel{\bL}{\otimes}_{\dR} \mP_n: \D(\dR)\lto\D(\dE)\quad\text{and}\quad
\bR\Hom_{\dE}(\mP_n, -):\D(\dE)\lto\D(\dR)
\]
that are adjoint to each othetr. The DG $\dE$\!--module $\mP_n$ is perfect and, hence, the derived functor $(-)\stackrel{\bL}{\otimes}_{\dR} \mP_n$  sends perfect modules to perfect ones.
Thus, the $\dR\hy\dE$\!--bimodule $\mP_n$ gives a quasi-functor $\mR: \prfdg\dR\to \prfdg\dE$ whose homotopy functor is exactly the functor
\[
(-)\stackrel{\bL}{\otimes}_{\dR} \mP_n: \prf\dR\lto\prf\dE.
\]

On the other hand, denote by $\D_{\cf}(\dR)\subset\D(\dR)$ the subcategory consisting of all cohomologically finite DG $\dR$\!--modules,
i.e all such $\mT$ that $\dim_{\kk}\bigoplus_i  H^i(\mT)<\infty.$ The DG functor $\dHom_{\dE}(\mP_n, -)$ sends all cohomologicaly finite DG $\dE$\!--modules to cohomologically finite DG $\dR$\!--modules and, hence,  it induces the functor
\[
\bR\Hom_{\dE}(\mP_n, -):\D_{\cf}(\dE)\lto\D_{\cf}(\dR).
\]

Let us consider the DG $\dE$\!--modules $\mK_i, 1\le i\le n$ from Theorem \ref{maintech} and denote by $\mK$ the direct sum $\bigoplus_{i=1}^{n}\mK_i.$ We can introduce a new finite-dimensional DG algebra $\dK=\dEnd_{\dE}(\mK).$ The following theorem allows us to describe the DG category $\prfdg\dE$ and a relation between DG categories  $\prfdg\dE$ and $\prfdg\dR.$

\begin{theorem}\label{DGalgebra}
Let $\dR=(R, \dr)$ be a finite-dimensional DG algebra of index nilpotency $n$ and let
$\dE=\dEnd_{\dR}(\bigoplus_{p=1}^{n} \mM_p)$ be the DG endomorphism algebra defined above. Then the following properties hold:
\begin{enumerate}
\item[1)] The functors  $(-)\stackrel{\bL}{\otimes}_{\dR} \mP_n: \prf\dR\to\prf\dE,\quad \D(\dR)\to\D(\dE)$ are fully faithful.
\item[2)] The triangulated category $\prf \dE$ has a full semi-exceptional collection of the form
\[
\prf\dE=\langle E_1,\dots, E_n\rangle,
\]
where the subcategories $\langle E_i\rangle$ are generated by DG $\dE$\!--modules $\mK_i$ from Theorem \ref{maintech}.
\item[3)] The DG category $\prfdg\dE$ is quasi-equivalent to $\prfdg\dK,$ where $\dK=\dEnd_{\dE}(\bigoplus_{i=1}^{n}\mK_i).$

\item[4)] The triangulated category $\prf \dE$ is regular and there is an equivalence
\[
\prf\dE\stackrel{\sim}{\lto}\D_{\cf}(\dE).
\]
\item[5)] If the semisimple part $\dS_{+}=\dR/\rde$ is separable, then the DG algebra $\dE$ is smooth.
\item[6)] If $\dR$ is smooth, then $\prf\dR$ is admissible in $\prf\dE.$
\end{enumerate}
\end{theorem}
\begin{proof}
1) It follows from the fact that $\mP_n$ is a perfect as DG $\dE$\!--module and  $\dEnd_{\dE}(\mP_n)\cong\dR.$ It is a standard argument from `devissage' (see, e.g., \cite[4.2]{Ke}, \cite[Prop. 1.15]{LO}).

2) It directly follows from Theorem \ref{maintech}. Indeed, the subcategories $\langle \mK_i\rangle$ are semi-orthogonal by property 3) of Theorem \ref{maintech} and give a semi-orthogonal decomposition
\[
\prf\dE=\langle \langle \mK_1\rangle, \ldots, \langle \mK_n\rangle\rangle
\]
because, by property 2), they generate the whole category. All DG algebras $\dEnd_{\dE}(\mK_i)$ are semisimple by property 4). This implies that every subcategory
$\langle \mK_i\rangle$ has a semi-exceptional generator $E_i,$ i.e. $\langle \mK_i\rangle=\langle E_i\rangle$ (see Remark \ref{semi-ex}). Thus, we obtain a full semi-exceptional collection of the form $\langle E_1,\dots, E_n\rangle$ in the triangulated category $\prf\dE.$

3) Consider the DG $\dE$\!-module $\mK=\bigoplus_{i=1}^{n}\mK_i.$ It is h-projective by 1) of Theorem \ref{maintech} and it is generates the category
$\prf\dE.$ Take the DG algebra $\dK=\dEnd_{\dE}(\mK).$ By the same argument from `devissage' as in 1) the DG categories $\prfdg\dE$ and $\prfdg\dK$ are quasi-equivalent (see, e.g., \cite[Prop. 1.15]{LO}).

4) We showed that the category $\prf \dE$ has a full semi-exceptional collection. It gives a semi-orthogonal decomposition with semisimple summands $\langle E_i\rangle.$
Thus, we obtain that $\prf\dE$ is regular, because all summands are regular  (see Section \ref{gluDNS} and proof in \cite[Prop. 3.20]{O_glue}).

Since $\prf \dE$ is proper and  regular, it is saturated by Theorem \ref{saturated}. This means that any  cohomological functor  from the category $\prf \dE^{\op}$ to the category of finite-dimensional
vector spaces is representable. But any object from the subcategory $\D_{\cf}(\dE)$ gives  such a functor and, hence, any object from $\D_{\cf}(\dE)$ is isomorphic to
an object from $\prf\dE.$ Therefore, the natural inclusion  $\prf\dE{\lto} \D_{\cf}(\dE)$ is an equivalence.

5) If the semisimple part $\dS_{+}=\dR/\rde$ is separable, then DG algebras $\dEnd_{\dE}(\mK_i)$ are separable too by 6) of Theorem \ref{maintech}. Hence, the full semi-exceptional collection $\langle E_1,\dots, E_n\rangle$ in the triangulated category $\prf\dR$
is separable. Thus, by Corollary \ref{sep_coll}, the DG algebra $\dE$ is $\kk$\!--smooth.

6) It is proved in \cite{Lu} that a small smooth DG category is regular. Hence, the triangulated category $\prf\dR$ is proper and regular.
By Proposition \ref{admissible} it is admissible in $\prf\dE.$
\end{proof}

The DG algebras $\dR$ and $\dE$ define two proper derived noncommutative schemes $\dX=\prfdg\dR$ and $\dY=\prfdg\dE,$ respectively.
The $\dR\hy\dE$\!--bimodule $\mP_n,$ which is a quasi-functor, gives a morphism of noncommutative schemes
$\mf: \dY\to\dX.$ Since the noncommutative scheme $\dY$ is regular (or even smooth), the morphism $\mf$ can be considered as an example of a resolution of singularities  of the noncommutative scheme $\dX=\prfdg\dR$
(see Definition 2.11 \cite{O_alg}).

For further, we will need a  definition of a pretriangulated DG category.
At first, to any DG category $\dA$ we can attach a DG category $\dA^{\ptr}$ which is called a {\sf pretriangulated hull}
(see \cite{BK2, Ke2}).
The DG category $\dA^{\ptr}$ is obtained by  adding to $\dA$
all shifts, all cones of all morphisms, cones of morphisms between cones and so on.
There is a canonical DG functor
 $\dA^{\ptr}\to \Mod\dA$ which is fully faithful, and under such embedding
$\dA^{\ptr}$ is equivalent to the DG category $\SFf\dA$ of finitely generated semi-free DG $\dA$\!--modules.
The pretriangulated hull $\dA^{\ptr}$ is  small for any small DG category $\dA,$ and it forms a small version of the DG category $\SFf\dA,$ which is essentially small.
A DG category $\dA$ is called {\sf pretriangulated}  if the inclusion $\dA\to\dA^{\ptr}$
is a quasi-equivalence.
This property is equivalent to the property for the homotopy category $\Ho(\dA)$ to be a triangulated
subcategory in $\Ho(\Mod\dA).$
The DG category $\dA^{\ptr}$ is already pretriangulated and $\Ho(\dA^{\ptr})$ is a triangulated category.
In the case when the DG category $\dA$ is pretriangulated and $\Ho(\dA)$ is idempotent complete,
the Yoneda functor $\mh^{\bullet}: \dA \to\prfdg\dA$ is a quasi-equivalence and the homotopy functor
$h: \Ho(\dA)\to \prf\dA$ between the triangulated categories is an equivalence.

Now we can formulate a corollary from the previous theorem which provides a complete characterization of all derived noncommutative schemes $\dX=\prfdg\dR$ obtained from a finite-dimensional DG algebra $\dR$ with a separable semisimple part $\dR/\rde.$

\begin{corollary}\label{DGinclus}
Let $\dA$ be a small $\kk$\!--linear DG category.
The following properties are equivalent:
\begin{enumerate}
\item[1)] The DG category $\dA$ is quasi-equivalent to $\prfdg\dR,$ where $\dR$ is a finite-dimensional  DG $\kk$\!--algebra for which DG
$\kk$\!--algebra $\dR/\rde$ is separable.
\item[2)] The DG category $\dA$ is quasi-equivalent to a full DG subcategory $\dC$ of a pretriangulated $\kk$\!--linear DG category $\dB$ such that
\begin{enumerate}
\item[a)]
the homotopy category $\H^0(\dB)$  is a proper idempotent complete triangulated category possessing a full separable semi-exceptional collection,
\item[b)] the homotopy category $\H^0(\dC)\subseteq\H^0(\dB)$ is an idempotent complete triangulated subcategory admitting  a classical generator.
\end{enumerate}
Moreover, the DG algebra $\dR$ is smooth if and only if $\H^0(\dC)\subseteq \H^0(\dB)$ is admissible.
\end{enumerate}
\end{corollary}
\begin{proof}
$1)\Rightarrow 2)$ It directly follows from Theorems \ref{DGalgebra} and \ref{maintech}.
Let $\dA\cong\prfdg\dR,$ where $\dR$ is a finite-dimensional  DG $\kk$\!--algebra with a separable semisimple part $\dR/\rde.$
Consider the DG algebra $\dE=\dEnd_{\dR}(\bigoplus_{p=1}^{n} \mM_p)$ as above. The category $\prf\dE$ is proper and idempotent complete.
It also has a full semi-exceptional collection by property 2). By property 5) of Theorem \ref{maintech} this collection is separable.
By property 1) of Theorem \ref{DGalgebra} the quasi-functor, which is defined by the bimodule $\mP_n,$ establishes a quasi-equivalence between $\prfdg\dR$ and a full DG subcategory of $\prfdg\dE,$ the homotopy category of which is equivalent to $\prf\dR.$ Thus its homotopy category is  an idempotent complete triangulated subcategory  of $\prf\dE$ admitting a classical generator.

$2)\Rightarrow 1)$
Consider the objects $\mE_1,\ldots, \mE_k$ in the DG category $\dB$ which form a full separable semi-exceptional collection
in the homotopy category $\H^0(\dB).$ Denote by $\dE$ the DG endomorphism algebra $\dEnd_{\dB}(\mE),$ where
$\mE=\bigoplus_{i=1}^{k}\mE_i.$ The object $\mE$ is a classical generator for $\H^0(\dB).$ Since
$\H^0(\dB)$ is idempotent complete, the DG category $\dB$ is quasi-equivalent to $\prfdg\dE$ (see, e.g., \cite{Ke, Ke2}, \cite[Prop. 1.15]{LO}).
By Corollary \ref{sep_coll} the DG $\kk$\!--algebra $\dE$ is smooth and it is quasi-isomorphic to a finite-dimensional DG algebra.
Its semisimple part $\dE/\mJ_{+}$ is separable.

Let us consider a classical generator $\mR\in\dC$ for the triangulated subcategory $\H^0(\dC)\subseteq\H^0(\dB).$
By the same reason as above, the DG category $\dC$ is equivalent to $\prfdg\dR,$ where $\dR=\dEnd_{\dB}(\mR).$
By Proposition \ref{obj}, the DG algebra $\dR$ is quasi-isomorphic to a finite-dimensional DG algebra.
We can assume that $\dR$ is finite-dimensional by itself and we have to argue that the DG $\kk$\!--algebra $\dR/\rde$ is separable.
As in Proposition \ref{obj} the object $\mR$ is homotopy equivalent to a direct summand of a finitely generated semi-free DG $\dE$\!--module $\mN.$
Hence, the semisimple part $\dR/\rde$ is a direct summand of a semisimple part of the DG endomorphism algebra
$\dEnd_{\dB}(\mN).$ On the other hand, the semisimple part of the  DG algebra $\dEnd_{\dB}(\mN)$ is covered by
the semisimple part of a free DG $\dE$\!--module. Since the semisimple part of $\dE$ is separable, the DG algebra $\dR/\rde$ is separable too.

 Suppose that $\dR$ is smooth. Hence, $\prf\dR$ is smooth and it is also regular (see \cite{Lu}). Therefore, the triangulated subcategory $\H^0(\dC)$ is proper and regular.
By Proposition \ref{admissible} it is admissible in $\H^0(\dB).$
Inverse, if $\H^0(\dC)\subseteq \H^0(\dB)$ is admissible, then $\dB$ is equivalent to a gluing of $\dC$ and its left (or right) orthogonal.
Since $\dB$ is smooth, then $\dC$ is smooth by Proposition \ref{smooth_glue}.
This implies that $\prfdg\dR$ is also smooth as well as the DG algebra $\dR.$
\end{proof}

The following corollary is an answer on a question of Mikhail Khovanov.
\begin{corollary}
Let $\dR$ be a smooth finite-dimensional DG algebra. Then $K_0(\prf\dR)$ is a finitely generated free abelian group.
\end{corollary}
\begin{proof}
By Theorem \ref{DGalgebra}, the triangulated category $\prf\dR$ can be realized as an admissible subcategory in the triangulated category
$\prf\dE$ that has a full semi-exceptional collection. The Grothendieck group $K_0(\prf\dE)$ is a finitely generated free abelian group.
Thus, the group $K_0(\prf\dR)$ is also  a finitely generated free abelian group as a direct summand of $K_0(\prf\dE).$
\end{proof}

\begin{problem}{\rm
For any smooth finite-dimensional DG algebra $\dR$ to find the rank of the free abelian group $K_0(\prf\dR)$
(see also Conjecture \ref{phantom}).
}
\end{problem}

\section{Geometric realizations of finite-dimensional DG algebras and examples}

\subsection{Geometric realizations}
In this section we are going to discuss geometric realizations for finite-dimensional DG algebras.
It will be shown that all of them (with a separable semisimple part) appear as DG endomorphism algebras of perfect complexes on smooth projective schemes
with full separable semi-exceptional collections and, moreover, this fact gives a complete characterization of such class of DG algebras.

A {\sf geometric realization} of a  derived noncommutative scheme $\dX=\prfdg\dR$ consists of 
 a usual commutative scheme $Z$ and a localizing subcategory $\L\subseteq \D_{\Qcoh}(Z),$  the natural enhancement $\dL$ of which
is quasi-equivalent
to the DG category $\SF\dR$ (see \cite[Definition 2.17]{O_alg}).
In other words a geometric realization is a fully faithful functor $\D(\dR)\to \D_{\Qcoh}(Z)$ that preserves all direct sums and is defined on the level of DG categories.
Since the category $\D(\dR)$ is compactly generated and the inclusion functor $\D(\dR)\to \D_{\Qcoh}(Z)$ commutes with all direct sums, there is a right adjoint to the inclusion functor by Brown representability theorem (see \cite[9.1.19]{Ne_book}).
Thus, the category $\D(\dR)$ can be obtained as a Verdier quotient of the triangulated category $\D_{\Qcoh}(Z).$

The most important and interesting class of  geometric realizations is related to  quasi-functors $\mF:\prfdg\dR\to \prfdg Z$
for which the induced homotopy functor $\prf\dR\to\prf Z$ is fully faithful.
In this case the subcategory $\L$ is compactly generated by  objects that are perfect complexes on $Z$
and, hence, there is an inclusion of the subcategories of compact objects $\L^c\cong\prf\dR\subset \prf Z.$
Such geometric realizations will be called {\sf {plain}}. In the case when the subcategory
$\prf\dR\subset\prf Z$ is admissible, the geometric realization will be called {\sf pure} (see \cite[Sect. 2.4.]{O_alg}).

Let $X$ and $Y$ be two usual irreducible smooth projective $\kk$\!--schemes and let $\mE\in\prfdg(X\times_{\kk} Y)$ be a perfect complex on the product $X\times_{\kk} Y.$
Consider the DG category which is  the gluing $(\prfdg X) \underset{\mE}{\oright}(\prfdg Y).$ This DG category defines a derived noncommutative scheme which we denote by $\dZ:=X\underset{\mE}{\oright} Y.$   Taking into account Propositions \ref{smooth_glue},
we see that the derived noncommutative scheme $\dZ$ is proper and smooth. It is natural to ask about existence of  geometric realizations for
such noncommutative schemes. The following theorem is proved in \cite{O_glue}.

\begin{theorem}\label{main}\cite[Th. 4.11]{O_glue}
Let $X$ and $Y$ be smooth irreducible projective schemes over a field $\kk$ and let $\mT$ be a perfect complex
on the product $X\times_{\kk} Y.$ Let $\dZ=X \underset{\mT}{\oright} Y$ be the derived noncommutative scheme obtained by the gluing of $X$ and $Y$ via $\mE.$
Then there exist a smooth projective scheme $V$ and a quasi-functor  $\mF: \prfdg\dZ\hookrightarrow\prfdg V$ which give a pure geometric realization
for the noncommutative scheme $\dZ.$
\end{theorem}

The proof of this theorem is constructive and it was  mentioned in Remark 4.14 of \cite{O_glue}
that actually  in this case the triangulated category $\prf V$ has a semi-orthogonal decomposition
$
\prf V=\langle \N_1,\dots \N_k\rangle,
$
where all $\N_i$ are equivalent to one of the following four categories: namely,
$\prf\kk,\; \prf X,\; \prf Y,$ and  $\prf (X\times_{\kk} Y).$

This statement can be extended   to the case of geometric derived noncommutative schemes.
Let $X$ be a usual smooth projective scheme and let $\dN\subset\prfdg X$ be a full pretriangulated DG subcategory.
Suppose that the homotopy triangulated  category $\N=\Ho(\dN)$ is admissible in $\prf X.$
By Proposition \ref{smooth_glue}, we obtain that the derived noncommutative scheme $\dN$
is smooth and proper. By definition, it already appears with a pure geometric realization.
Such derived noncommutative scheme will be called geometric.

\begin{theorem}\label{main2}\cite[Th. 4.15]{O_glue}
Let the DG categories $\dN_i,\; i=1,\dots, n$ and the smooth projective schemes $X_i,\; i=1,\dots, n$ be as above.
Let $\dZ=\prfdg\dR$ be a proper derived noncommutative scheme with full embeddings
 of the DG categories $\dN_i\subset \prfdg\dR$ such that $\prf\dR$ has a semi-orthogonal decomposition of the form
$
\langle\N_1, \N_2,\dots, \N_n\rangle,
$
where $\N_i=\Ho(\dN_i).$ Then there exist a smooth projective scheme $V$ and a quasi-functor  $\mF: \prfdg\dZ\hookrightarrow\prfdg V$ which gives a pure geometric realization
for the noncommutative scheme $\dZ.$
\end{theorem}

The derived noncommutative scheme $\dZ$ is also smooth as a gluing of smooth DG categories via perfect DG bimodules.
Now Theorem \ref{main2} shows us that the class of all geometric smooth and proper derived noncommutative schemes is closed under gluing operation
with respect to perfect bimodules.

Now we consider derived noncommutative schemes that are related to finite-dimensional DG algebras with sparable semi-simple part and will discuss existing of geometric realizations for such derived noncommutative schemes. The following theorem say us that any such DG $\kk$\!--algebra $\dR$ can be realized as a DG endomorphism algebra of a perfect complex $\mE$ on a smooth projective scheme $V$ with a full separable semi-exceptional collection.

\begin{theorem}\label{algebra}
Let $\dR$ be a finite-dimensional DG $\kk$\!--algebra. Suppose that the semisimple DG $\kk$\!--algebra $\dS_{+}=\dR/\rde$ is  separable.
Then there are a smooth projective scheme
$V$ and a perfect complex $\mE\in \prfdg V$ which satisfy the following properties:
\begin{enumerate}
\item[1)] The DG algebra $\dEnd_{\prfdg V}(\mE)$  is quasi-isomorphic to $\dR.$
\item[2)] The DG category $\prfdg\dR$ is quasi-equivalent to a full DG subcategory $\dC$ of $\prfdg V.$
\item[3)] The category $\prf V$ has a full separable semi-exceptional collection.
\end{enumerate}
Moreover, if $\dR$ is smooth, then the homotopy category
$\H^0(\dC)$ is  admissible in $\prf V.$
\end{theorem}
\begin{proof}
At first, let us note that property 1) directly follows from properety 2).

Consider as above
the DG endomorphism algebra $\dE=\dEnd_{\dR}(\mM)$ of the DG $\dR$\!--module $\mM=\bigoplus_{p=1}^{n} \mM_p,$ where DG modules $\mM_p=\dR/\mJ_p$ with $p=1,\ldots, n$
were constructed in Section \ref{ausl}.
Taking in account Theorem \ref{DGalgebra} we can consider the DG algebra $\dE$ instead of $\dR$ and show that there exists
a smooth projective scheme
$V$ such that the properties 2)-3) hold for $\dE.$

By Theorem \ref{DGalgebra} the triangulated category $\prf\dE$ has a full semi-exceptional collection $(E_1,\dots, E_n)$ and a semi-orthogonal decomposition
of the form
\[
\prf\dE=\langle E_1,\dots, E_n\rangle
\]
where the subcategories $\langle E_i\rangle$ are generated by DG $\dE$\!--modules $\mK_i$ from Theorem \ref{maintech}.
Any subcategory $\langle E_i\rangle$ is an orthogonal sum of triangulated categories of the form $\prf D,$ where $D$ is a division $\kk$\!--algebra.
Thus, we can consider a new full semi-exceptional collection $\langle F_1, \ldots, F_N\rangle,$ which gives a refinement of the previous
semi-orthogonal decomposition, and such that each subcategory $\langle F_i\rangle$ is simple and equivalent to a triangulated category $\prf D_i,$
where $D_i$ is a division $\kk$\!--algebra. All these division $\kk$\!--algebras $D_i$ are also separable by property 5) of Theorem \ref{maintech}.

We will prove the theorem for $\dE$ proceeding by induction on $N$ that is the length of the full semi-exceptional collection $\langle F_1, \ldots, F_N\rangle.$
The base of induction is $N=1,$ and by Remark \ref{Severi-Brauer} we can take the Severy-Brauer variety $SB(D)$ and consider it as a variety
over the field $\kk.$ In this case the DG category $\prfdg D$ is quasi-equivalent to a full DG subcategory $\dC$ of $\prfdg SB(D)$
and the triangulated category $\prf SB(D)$ has a full separable semi-exceptional collection.

Now by Proposition 3.8 from \cite{O_glue} the DG category $\prfdg\dE$ is quasi-equivalent to a gluing of  DG subcategories $\prfdg\dE'$ and $\prfdg D_N,$ where
the first subcategory is generated by the collection $\langle F_1, \ldots, F_{N-1}\rangle$ while the second subcategory
is generated by $F_N.$ By induction hypothesis there is a smooth projective variety $V'$ such that properties 1)-3) hold for
the DG category $\prfdg\dE'.$
By Theorem \ref{main2} there exist a smooth projective scheme $V$ and a quasi-functor  $\mF: \prfdg\dE\hookrightarrow\prfdg V$ which establishes a quasi-equivalence of $\prfdg\dE$ with a full DG subcategory $\dC\subseteq\prfdg V.$
Moreover, as it was shown in the proof of Theorem \ref{main2} in \cite{O_glue} the DG category $\prfdg\dE$ is realized
as a full DG subcategory of a gluing $\prfdg V' \underset{\mT}{\oright}\prfdg SB(D_N)$ via  some perfect DG bimodule $\mT.$ After that we can apply Theorem \ref{main}
and construct a smooth projective scheme $V$ such that the DG category $\prfdg V' \underset{\mT}{\oright}\prfdg SB(D_N)$ is quasi-equivalent
to a full DG subcategory of $\prfdg V.$

Furthermore, by Remark 4.14 from \cite{O_glue} the category $\prf V$ from  Theorem \ref{main} has a semi-orthogonal decomposition of the form
$
\prf V=\langle \N_1,\dots \N_k\rangle,
$
where each $\N_i$ is equivalent to one of the following four categories
$\prf\kk,\; \prf V',\; \prf SB(D_N),$ and  $\prf (V'\times_{\kk} SB(D_N)).$
But all these categories have full separable semi-exceptional collections.
This implies that the category $\prf V$ also has a full separable semi-exceptional collection.
\end{proof}

The following corollary is an analogue of Corollary \ref{DGinclus} and it also gives a characterization of derived noncommutative schemes obtained from  finite-dimensional DG algebras with separable semisimple part but in geometric terms.

\begin{corollary}\label{DGinclus_smooth}
Let $\dA$ be a small $\kk$\!--linear DG category.
The following properties are equivalent:
\begin{enumerate}
\item[1)] The DG category $\dA$ is quasi-equivalent to $\prfdg\dR,$ where $\dR$ is a finite-dimensional  DG algebra for which $\dR/\rde$ is separable.
\item[2)] The DG category $\dA$ is quasi-equivalent to a full DG subcategory $\dC$ of a DG category $\prfdg V,$ where
\begin{enumerate}
\item[a)] $V$ is a smooth projective scheme for which $\prf V$ has a full separable semi-exceptional collection.
\item[b)] the homotopy category $\H^0(\dC)\subseteq\prf V$ is an idempotent complete triangulated subcategory admitting  a classical generator.
\end{enumerate}
\end{enumerate}
Moreover, the DG algebra $\dR$ is smooth if and only if $\H^0(\dC)\subseteq\prf V$ is admissible.
\end{corollary}
\begin{proof}
$1)\Rightarrow 2)$ It directly follows from Theorems \ref{algebra}.
Let $\dA\cong\prfdg\dR,$ where $\dR$ is a finite-dimensional  DG algebra with separable semisimple part $\dR/\rde.$
By Theorem \ref{algebra} there is a smooth projective scheme $V$ and a perfect complex
$\mE$ such that $\dEnd_{\prfdg V}(\mE)\cong\dR.$ Thus, the DG category $\prfdg\dR$ is quasi-equivalent to a full DG subcategory $\dC$ of $\prfdg V.$ Hence,  the homotopy category $\H^0(\dC)\subseteq\prf V$ is  an
idempotent complete triangulated subcategory admitting  a classical generator.
Moreover, the category $\prf V$ has a full separable semi-exceptional collection by property 3) of Theorem \ref{algebra}.

$2)\Rightarrow 1)$  is a particular case of Corollary \ref{DGinclus}.
In the end of Corollary \ref{DGinclus} it is also shown that the smoothness of $\dR$ is equivalent to the property for the subcategory $\H^0(\dC)$ to be admissible.
\end{proof}

\begin{remark}
{\rm
If the base field $\kk$ is perfect, then  any semisimple algebra is separable.
Thus, the results of this section can be applied  to all finite-dimensional DG algebras over a perfect field $\kk.$
}
\end{remark}

\subsection{Conjectures and questions}
It is very interesting to understand and to describe smooth projective varieties, for which the triangulated category of perfect complexes is equivalent to
$\prf\dR,$ where $\dR$ is a finite-dimensional DG algebra. Even for algebraically closed field $\kk$ of characteristic zero we are quite far from a solving of this problem.
It is rather natural to hypothesize that in this case such a variety is rational and, moreover, it has a stratification by a rational strata.

\begin{conjecture}
Let $Z$ be a projective scheme such that the DG category $\prfdg Z$ is quasi-equivalent  to $\prfdg\dR,$ where $\dR$ is
a finitely-dimension DG algebra. Then $Z$ is geometrically rational. Moreover,
the scheme $\widebar{Z}=Z\otimes_{\kk}\widebar{\kk}$ has a stratification $\widebar{Z}=\bigcup_{i=1}^{k} Y_i$ such that any $Y_i$ is an open subset of $\AA^{n_i}.$
\end{conjecture}
It is possible that the last statement gives a complete description of such varieties. Moreover, we can ask about such stratification
that all $Y_i\subseteq \AA^{n_i}$ also have complements with stratifications described above. It is also possible that some reasonable stratifications exist over a base field if we ask that $Y_i$ are open subsets in $\AA^{n_i}_{\ff_i},$ where $\ff_i$ are finite extensions of $\kk.$

It is also interesting and very useful to know about existing (or not existing) of phantoms of the form $\prfdg\dR,$ where
$\dR$ is finite-dimensional.
We recall that a smooth and proper DG category
$\prfdg\dR$ is called  a {\sf quasi-phantom} if the Hochschild homology are trivial, i.e.
$\mathrm{HH}_{*}(\prfdg\dR)=0,$ and $K_0(\prfdg\dR)$ is a finite abelian group.
It is  called a {\sf phantom} if, in addition, $K_0(\prfdg\dR)=0$ (see \cite{GO}, \cite[Def. 3.18]{O_alg}).
It was proved in \cite{GO} that geometric phantoms exist. However, there are some arguments to think that all (quasi)-phantoms are complicated as DG categories and they cannot be of such form
as $\prfdg\dR$ with a finite-dimensional DG algebra $\dR.$

\begin{conjecture}\label{phantom}
There are no phantoms of the form $\prfdg\dR,$ where $\dR$ is a smooth finite-dimensional DG algebra.
\end{conjecture}

We should note that at this moment we know many examples of smooth projective varieties for which
the triangulated category of perfect complexes have  full (semi-)exceptional collections. In this cases the DG category
$\prfdg X$ is quasi-equivalent to $\prfdg\dR$ with finite-dimensional $\dR.$ On the other hand, we do not know any smooth projective variety, for which
there is such an equivalence, but the category of perfect complexes does not have a full (semi-)exceptional collection.

\begin{question}
{\rm Is there a smooth projective variety $X$ such that $\prfdg X$ is quasi-equivalent to $\prfdg \dR,$ where
$\dR$ is a smooth finite-dimensional DG algebra, but $X$ does not have a full semi-exceptional collection.
}
\end{question}

Suppose that such a smooth projective variety $X$ exists. In this case we can consider different blow ups $\widetilde{X}\to X$ of this variety along smooth subvarieties also possessing full exceptional collections.
The results of \cite{Blow} allow us to assert that for any such $\widetilde{X}$ the DG categories $\prfdg \widetilde{X}$ is quasi-equivalent to $\prfdg \widetilde{\dR}$ with a
smooth finite-dimensional $\widetilde{\dR}$ as well. It is possible that some of such blow ups $\widetilde{X}$ already have full exceptional collections in spite of the variety $X$ doesn't have.

\begin{question}
{\rm  Is there a smooth projective variety $X$ that does not have a full semi-exceptional collection but some its blow up
$\widetilde{X}\to X$ already has such a collection.
}
\end{question}

\subsection{Examples of geometric realizations for some DG algebras}
In the end of the paper we consider some examples of Auslander construction for DG algebras and geometric realizations that appear here.

Let $V$ be  a vector space of dimension $n.$ Consider the exterior algebra $\Lambda^* V=\bigoplus_{q} \Lambda^q V$
as a graded algebra, where elements of $V$ have degree $1.$ Denote by $\dR$ the DG algebra $\Lambda^* V$ with the differential equal to $0.$
Let us apply the Auslander construction to the DG algebra $\dR,$ i.e. let us consider the DG algebra
$\dE$ that is equal to $\dEnd_{\dR}(\mM),$ where $\mM=\bigoplus_{p} \dR/\mJ_p$ as in the beginning of Section \ref{ausl}.
As it was proved in Theorem \ref{DGalgebra} the category $\prf\dE$ has a full semi-exceptional collection that is formed by
the objects $\mK_i, 1\le i\le n+1$ from Theorem \ref{maintech}.
The objects $\mK_i$ are h-projective and we denote by $\dK$ the DG algebra $\dEnd_{\dE}(\mK),$ where $\mK$ is the sum
$\bigoplus_{i=1}^{n+1}\mK_i.$
By property 3) of Theorem \ref{DGalgebra} the DG categories $\prfdg\dE$ and $\prfdg\dK$ are quasi-equivalent.
In this case it is more convenient to consider another derived Morita equivalent DG algebra $\dK'$ that is by definition
$\dEnd_{\dE}( \bigoplus_{i=1}^{n+1}\mK_i[i-1]).$

Let us calculated the DG algebra $\dK'.$
It is easy to see that all objects $\mK_i$ are exceptional, i.e. $\dEnd_{\dE}(\mK_i)$ are quasi-isomorphic to $\kk$ for all $i.$
Moreover, we can easily checked that for any pair $i<j$ the complex $\dHom_{\dE}(\mK_i, \mK_j)$ has only one nontrivial cohomology in 0-term and it is isomorphic to $\Lambda^{j-i}V\oplus \Lambda^{j-i-1}V.$
This implies that the DG algebra $\dK'$ is quasi-isomorphic to a usual algebra.
Moreover, we can consider a vector space $U$ that is equal to $V\oplus\kk.$ There are canonical isomorphisms $\Lambda^{j-i} U\cong \Lambda^{j-i}V\oplus \Lambda^{j-i-1}V.$
Let us now consider the projective space $\PP(U)$ and the following full strong exceptional collection on it
\[
\langle \cO, \T(-1), \Lambda^2\T(-2), \dots, \Lambda^n\T(-n)=\cO(1)\rangle,
\]
where $\T$ is the tangent bundle on $\PP(U).$ It is not difficult to check that the endomorphism algebra of this exceptional collection
$A=\End(\bigoplus_{i=0}^{n}\Lambda^i\T(-i))$
is quasi-isomorphic to the DG algebra $\dK'.$ Thus, we obtain the following quasi-equivalences between DG categories
\[
\prfdg\dE\cong\prfdg\dK\cong\prfdg\dK'\cong\prfdg A\cong\prfdg \PP(U).
\]

Finally, we can note that the decomposition $U=V\oplus\kk$ gives a fixed point $p\in\PP(U).$ It can be checked that the
natural functor $\prf\dR\to\prf\dE\cong\prf \PP(U)$ from Theorem \ref{DGalgebra} sends the DG algebra $\dR$ as the right module to the structure sheaf $\cO_p\in\prf\PP(U) $ of the point $p\in\PP(U).$

\begin{proposition} The Auslander construction from Theorem \ref{DGalgebra} applied to the DG algebra $\dR,$ which is the graded exterior algebra
$\Lambda^* V,$ with $\deg V=1, \dim V=n$ and with the trivial differential, provides a full embedding $\prf\dR$ to $\prf\PP^n$ sending
the algebra $\dR\in \prf\dR$ to the structure sheaf $\cO_{p}\in\prf \PP^n$ of a point $p\in\PP^n.$
\end{proposition}

Let us consider another graded algebra that will be considered as the DG algebra with zero differential. Denote by $\dR$ the DG algebra $(\kk[x]/x^n, 0),$ where $\deg x=1.$
Let us apply the Auslander construction to it and consider the DG algebra
$\dE$ that is equal to $\dEnd_{\dR}(\mM),$ where $\mM=\bigoplus_{p=1}^n \kk[x]/x^p.$
The category $\prf\dE$ has a full semi-exceptional collection that is actually exceptional. This collection is formed by
the objects $\mK_i, 1\le i\le n$ from Theorem \ref{maintech}.
Denote by $\dK'$ the DG algebra $\dEnd_{\dE}(\mK'),$ where $\mK'$ is the sum
$\bigoplus_{i=1}^{n}\mK_i[i-1].$
As above the DG category $\prfdg\dE$ is quasi-equivalent to $\prfdg\dK'.$

It can be shown that the DG algebra $\dK'$ is quasi-isomorphic to a usual algebra and this algebra can be described as the path algebra $\kk[Q]$ of the following quiver with relations
\begin{equation*}\label{Ising}
Q=\Bigl(
\xymatrix{
\underset{1}{\bullet} \ar@<1ex>[r]^{a_1}\ar@<-1ex>[r]_{b_1} &
\underset{2}{\bullet} \ar@<1ex>[r]^{a_2} \ar@<-1ex>[r]_{b_2} &
\underset{3}{\bullet} \ar@{.}[r] &
\underset{n-1}{\bullet} \ar@<1ex>[r]^{a_{n-1}}\ar@<-1ex>[r]_{b_{n-1}} &
\underset{n}{\bullet}
}
\Big|\quad  a_{i+1} b_i=0,\; b_{i+1} a_i=0
\Bigr).
\end{equation*}
Thus, we obtain that the DG category
$
\prfdg\dE
$
is quasi-equivalent to $\prfdg \kk[Q].$
The quiver $Q$ appears as a  directed Fukaya category of a double covering
$C\to\AA^1$ branched along $n$ points, so
that the total space is a hyperelliptic curve of genus $g=\left[\frac{n-1}{2}\right]$ without one or two points (see \cite{Se}).
We are not going to discuss a geometric realization for this quiver here, but the case of $n=3$ is considered in the paper \cite{O_q}.

\end{document}